\documentclass[11pt,a4paper]{article}
\usepackage[utf8]{inputenc}
\usepackage[T1]{fontenc}
\usepackage[english]{babel}
\usepackage[onehalfspacing]{setspace}
\usepackage[a4paper, textwidth=15cm, bottom=4cm]{geometry}

\usepackage{amsthm}
\usepackage{thmtools}

\usepackage{amsmath,mathtools,amssymb,amstext}
\usepackage[hidelinks]{hyperref}
\usepackage{cleveref}
\usepackage{graphicx}
\usepackage{multirow}

\usepackage{mathrsfs}
\usepackage[title]{appendix}
\usepackage{xcolor}
\usepackage{textcomp}
\usepackage{manyfoot}
\usepackage{booktabs}
\usepackage{algorithm}
\usepackage{algorithmicx}
\usepackage{algpseudocode}
\usepackage{authblk}
\usepackage{listings}

\usepackage[T1]{fontenc}
\usepackage[english]{babel}

\usepackage{amsfonts}
\usepackage{dsfont}
\usepackage{xcolor}
\usepackage{comment}
\usepackage{tikz}
\usepackage{caption}

\usepackage{subcaption}
\usepackage{commath}
\usepackage{pgfplots}
\usepackage{scalerel}
\usetikzlibrary{patterns}
\usepackage[outline]{contour}
\usepackage{ifthen}
\usepackage{xkeyval}
\usepackage{calc}
\usepackage{extarrows}

\usepackage{listings}

\usepackage[utf8]{inputenc}

\contourlength{0.4pt}

\numberwithin{equation}{section}

\theoremstyle{definition}
\newtheorem{theorem}{Theorem}[section]

\newtheorem{lemma}[theorem]{Lemma}

\newtheorem{definition}[theorem]{Definition}

\newcommand{\y}[1]{\lVert #1 \rVert}

\newcommand{\Y}[1]{\left\lvert#1\right\rvert}

\newcommand{\under}[2]{\underset{#1}{\underbrace{#2}}}

\newcommand{\R}{\mathbb{R}}
\newcommand{\Z}{\mathbb{Z}}

\newcommand{\Raeps}{\mathcal{R}_\varepsilon}

\newcommand{\Reps}{\mathcal{R}^*_\varepsilon}
\newcommand{\Repsu}{\mathcal{R}^*_\varepsilon u_\varepsilon}

\newcommand{\Eeps}{\mathcal{E}_\varepsilon}
\newcommand{\Eepsut}{\mathcal{E}_\varepsilon (t, u_\varepsilon , z_\varepsilon )}

\newcommand{\ocweight}{\varpi}

\newcommand{\Jepsut}{\mathcal{E}^{V}_\varepsilon (u_\varepsilon)}

\newcommand{\Eloct}{\mathcal{E}^{loc}_\varepsilon (u_\varepsilon ,z_\varepsilon )}

\newcommand{\asarrow}{\xlongrightarrow[\text{a.s.}]{\varepsilon \rightarrow 0} }
\newcommand{\epsarrow}{\xlongrightarrow{\varepsilon \rightarrow 0} }

\newcommand{\coeff}{\sigma_{\varepsilon, x,  y}}

\newcommand{\estconst}{\nu}

\usepackage{xparse}

\NewDocumentCommand{\gagl}{ O{u} O{s} O{p} }
{\left[ #1 \right]_{W^{#2,#3}}
}
\NewDocumentCommand{\tgagl}{ O{u} O{s} O{2} O{Q} }
{\left[ #1 \right]_{W^{#2,#3}(#4)}
}

\newcommand{\eps}{\varepsilon}

\def\Xint#1{\mathchoice
{\XXint\displaystyle\textstyle{#1}}
{\XXint\textstyle\scriptstyle{#1}}
{\XXint\scriptstyle\scriptscriptstyle{#1}}
{\XXint\scriptscriptstyle\scriptscriptstyle{#1}}
\!\int}
\def\XXint#1#2#3{{\setbox0=\hbox{$#1{#2#3}{\int}$}
\vcenter{\hbox{$#2#3$}}\kern-.5\wd0}}

\def\dashint{\Xint-}
\newcommand{\uproman}[1]{\uppercase\expandafter{\romannumeral#1}}

\definecolor{dblue}{RGB}{47,82,143}
\definecolor{dgreen}{RGB}{49,142,33}
\definecolor{grey}{RGB}{118,113,113}

\title{Homogenization of rate-independent elastoplastic spring network models with non-local 
random fields}
\date{}

\author[1,2]{Simone Hermann \thanks{simone.hermann@mathematik.uni-freiburg.de}}

\affil[1]{\footnotesize 
{Cluster of Excellence livMatS @ FIT – Freiburg Center for Interactive
Materials and Bioinspired Technologies,
University of Freiburg,
Georges-Köhler-Allee 105, 79110 Freiburg, Germany}}
\affil[2]{Department for Applied Mathematics, University of Freiburg, Hermann-Herder-Str. 10, 79104 Freiburg, Germany}

\allowdisplaybreaks
\begin{document}

\maketitle


~


~

\begin{abstract}
We investigate the time-evolution of elastoplastic materials reinforced by randomly distributed long-range interactions. Starting from a 
rate-independent system on a discrete spring lattice that combines local linearized elasticity, gradient-regularized plasticity and stochastic non-local links modeling stiff fibers, we establish a discrete-to-continuum limit in the energetic formulation. We prove that as the lattice spacing tends to zero, an evolutionary solution of the discrete system converges to the unique energetic solution of a continuum limit problem. The limiting continuum model couples classical elastoplasticity with a non-local energy featuring fractional-order interactions  
that capture the homogenized influence of random long-range reinforcements. These results extend previous static homogenization studies by rigorously treating path-dependent dissipation and showing existence, uniqueness and Lipschitz continuity of the continuum evolutionary solution. The work therefore provides a mathematical foundation for simulating time-dependent mechanical response of fiber-reinforced composites with random architecture.
\end{abstract}

\newpage

\section{Introduction}
Nature has evolved remarkable structural solutions to mechanical challenges through biological composites which in many cases contain fibrous networks such as the peel of citrus fruits or the fibrous mesocarp of the coconut (\textit{Cocos nucifera}) \cite{jentzsch2022damage, seidel2010fruit, Coconut_paper_fibers}. 
A particularly outstanding example of this material class
that achieves exceptional damage protection properties is the peel of the pomelo (\textit{Citrus maxima}), which consists of a soft foam-like tissue and a network of vascular fiber bundles \cite{thielen2013structure}.
While such fiber-reinforced composites have been experimentally studied extensively, advanced mathematical models may help to provide a deeper understanding of the underlying energy dissipation mechanisms.
This motivates our investigation into discrete networks with random non-local interactions, which mimic the way stiff fibers transmit forces over long distances in a softer matrix.

A powerful mathematical tool for studying such materials is the homogenization of discrete models. 
Discrete conductance models with random coefficients have been studied in the context of homogenization across various frameworks~\cite{neukamm2018stochastic, Heida1, Heida2, HeidaHermann}. These models provide a versatile tool to describe heterogeneous materials with complex microstructures and to derive their effective macroscopic properties via homogenized limit models.

For many spring network models, homogenization leads to a local continuum description. However, recent works~\cite{Heida2, HeidaHermann} have demonstrated that, under suitable assumptions, such models can also exhibit non-local homogenized limits.
In particular, the work~\cite{HeidaHermann} developed a homogenization approach for stationary problems on discrete lattices with randomly distributed long-range interactions. The random structure was modeled by assigning i.i.d.\ Bernoulli variables to all pairs of points in the lattice, with the connection probability decaying with distance. A distinctive feature of~\cite{HeidaHermann} is that all non-local interactions are assumed to have the same strength. This leads to spatially non-stationary random fields and precludes the use of standard ergodic theorems typically employed in stochastic homogenization. Nevertheless, this approach provides a robust framework that captures both local and long-range interactions in a unified model, suitable for describing the effective elastic behavior of materials reinforced by long fibers.

However, many materials of practical relevance exhibit not only elastic but also plastic or dissipative behavior. In such cases, the material response is often well described by rate-independent systems, where the evolution depends on the loading path but not on the rate at which it is applied.
The mathematical theory of rate-independent systems, developed in the works of Mielke and collaborators~\cite{ mielkeRoubicek, MIELKE2}, provides a robust variational framework for these problems, combining stability conditions and energy balance with a dissipation potential.

The present paper advances the stationary stochastic homogenization results of~\cite{HeidaHermann} by developing an evolutionary, rate-independent model that incorporates linearized elastoplasticity, gradient plasticity, and hardening. In contrast to the more abstract and simplified setting of~\cite{HeidaHermann}, we formulate a detailed framework that captures genuine elastoplastic material behavior, enabling the description of time-dependent dissipative phenomena. This refined model is specifically tailored to represent the mechanical response and dissipation mechanisms observed in biological composites such as the pomelo peel, thus providing a direct link between mathematical analysis and experimentally accessible material properties.
Furthermore, this framework enables the computational prediction of effective material properties through numerical simulations that capture how random fiber distributions influence the macroscopic mechanical behavior of biological composites.
\section{Material model}
 
The remarkable toughness of materials such as pomelo peel or bio-inspired composites is often attributed to a network of long, fiber-like reinforcements bridging distant points inside a softer matrix.
A convenient mathematical idealization of this architecture is a spring lattice: each lattice node represents a material point, and edges encode load-carrying fibers.  Springs are taken to be linear Hookean for the elastic part of the response, while plastic slips evolve non-linearly through a convex dissipation potential.  
In our model, the matrix material is assumed to operate in the small-strain regime, so all local (nearest-neighbor) elastic interactions are described by linearized elasticity. Nonlinearity in the matrix response is introduced exclusively through the evolution of plastic deformations, modeled via a convex, rate-independent dissipation potential and gradient-regularized hardening.

In contrast, the long-range fiber connections are modeled as perfectly undeformable links: any relative displacement between connected points directly contributes to the energy without any elastic compliance. This reflects the physical observation that stiff fiber bundles in biological tissues, such as vascular bundles in pomelo peel, transmit forces over long distances with negligible extension compared to the surrounding matrix.

\subsection{Mathematical model}
This section establishes the mathematical framework for our analysis. We introduce the core definitions and formulate the key assumptions that are presumed to hold throughout the remainder of this paper. 

Let $Q\subset \mathbb{R}^d$, $d\geq 1$, be a bounded Lipschitz domain and $0 < \varepsilon < 1$. Then the underlying lattice is given by $Q_\varepsilon \coloneqq Q \cap \mathbb{Z}_\varepsilon^d$, where $\mathbb{Z}_\varepsilon^d \coloneqq \varepsilon \mathbb{Z}^d$.
All functions considered in the discrete setting are defined in the way $w_\varepsilon : \mathbb{Z}_\varepsilon^d \rightarrow \mathbb{R}^d$ with $w_\varepsilon (x) = 0$ for all $x \in \mathbb{Z}_\varepsilon^d \setminus Q_\varepsilon$.
Furthermore, we consider the following difference-type operators in the discrete setting:
\begin{align*}
    \partial_i^\varepsilon w_\varepsilon (x) &\coloneqq \frac{w_\varepsilon(x+\varepsilon e_i) - w_\varepsilon(x) }{\varepsilon} ,\\
    \nabla_\varepsilon w_\varepsilon (x) &\coloneqq \left(  \partial_1^\varepsilon w_\varepsilon (x),..., \partial_d^\varepsilon w_\varepsilon (x)  \right), \\
    \nabla^s_\varepsilon w_\varepsilon (x) &\coloneqq \frac{1}{2} \left( \nabla_\varepsilon w_\varepsilon(x) + (\nabla_\varepsilon w_\varepsilon (x) )^\top \right).
\end{align*}
Thus, $\nabla^s_\varepsilon$ denotes the symmetric part of the discrete gradient $\nabla_\varepsilon$. The standard notations $\nabla^s$ and $\nabla$ are used for the respective continuum operators.
As usual in elasticity theory, we denote the deformation at each grid point $x\in Q_\varepsilon$ with $u_\varepsilon(x) \in \mathbb{R}^d$. Then, the strain is given by $\nabla^s_\varepsilon u_\varepsilon $ and in the elastoplastic framework it is assumed that it can be additively decomposed into its elastic part $e^\mathrm{el}_\varepsilon$ and its plastic part $z_\varepsilon \in \mathbb{R}_{\text{sym}, 0} ^{d\times d}$, where 
\begin{align*}
    \R_{\text{sym}, 0} ^{d\times d} = \lbrace A \in \R^{d \times d} ~\vert ~ A = A^T ~\text{and}~ \text{tr}(A)=0 \rbrace.
\end{align*}
Thus, the elastic part turns into
\begin{align*}
    e_\varepsilon^\mathrm{el} = \nabla ^s_\varepsilon u_\varepsilon -z_\varepsilon.
\end{align*}
In our model, the internal energy of the matrix material is given by
\begin{align}\label{eq:Eepsloc}
    \Eloct   \coloneqq \varepsilon^{d} \sum_{x\in \Z_\varepsilon^d} e_\varepsilon^\mathrm{el}(x): \bold{A}(x): e_\varepsilon^\mathrm{el} 
  + z_\varepsilon(x) : \bold{H}(x) : z_\varepsilon(x) + \kappa \Y{\nabla_\varepsilon  z_\varepsilon (x)}^2 ,
\end{align}
where $\bold{A}$ and $\bold{H}$ are positive and symmetric fourth-order tensors and $\kappa > 0$. For isotropic materials it is
$\bold{A}(x): X=2\mu(x) X+\lambda(x)\text{tr}(X)\mathrm{Id}$, where $\lambda$ and $\mu$ are the Lamé constants.
In this work it is assumed that the minimal eigenvalues $\lambda_{\bold{A}}^\mathrm{min}, \lambda_{\bold{H}}^\mathrm{min}$ of the tensors $\bold{A}$ and $\bold{H}$ satisfy the relation 
\begin{align}
   0 < \lambda_{\bold{A}}^\mathrm{min} < \lambda_{\bold{H}}^\mathrm{min}.
\end{align}
The first term of the discrete local energy \eqref{eq:Eepsloc} accounts for energy contributions of purely elastic effects, whereas the second term describes the energy contribution due to hardening effects.
The quadratic gradient term $\kappa\lvert\nabla_\varepsilon z_\varepsilon\rvert^{2}$
suppresses microscopic oscillations and ensures compactness.

The long-range fiber connections are introduced via conductances $\ocweight $ between any grid points $x,y \in Q_\varepsilon$. These conductances are assumed to be i.i.d. Bernoulli random variables, i.e. $\ocweight (z_1,z_2) : \mathbb{Z}^d\times \mathbb{Z}^d \rightarrow \lbrace 0,1 \rbrace$ with probabilities for a connection given by
\begin{align*}
    \mathbb{P}\left( \ocweight (z_1,z_2) =1 \right)  &= 
\begin{cases}
   \Y{z_1-z_2} ^{-d-ps}& ~\text{if}~ \Y{z_1-z_2} >0, \\
0& ~\text{else}.
\end{cases}
\end{align*}
Here $d\in \mathbb{N}$, $s\in (0,1) $,  and
\begin{align*}
     p &\in \tfrac{2d}{d+2(s-1)}  
     & \text{if}~ d>2, \\
     p &\in   [2, \infty) &\text{else.}
\end{align*}

In the previous paper \cite{HeidaHermann}, the bound $\tfrac{2d}{d-2}$ on $p$ for $d>2$ was necessary for the embedding $H^1(Q)^d \hookrightarrow L^p(Q)^d$. In the current paper the stricter bound $\tfrac{2d}{d+2(s-1)}$ is needed, as it ensures the embedding $H^1(Q)^d \hookrightarrow W^{s,p}(Q)^d$ and hence the consistency of the space definition $W^{s,p}(Q)^d\cap H^1(Q)^d$ for the deformation component.
Furthermore, all parameters must satisfy the relation
\begin{align*}
    d > ps .
\end{align*}
The free energy of a single fiber with endpoints $x$ and $y$ is given by
\begin{align*}
   E_{xy} = \left( \Y{x+u(x)-y-u(y)} - \Y{x-y} \right)^2,
\end{align*}
where constants have been omitted for simplicity.
We restrict our study to the regime of small deformations so that the energy can be linearized as
\begin{align*}
    E_{xy} \approx \Y{\left( u(x)-u(y) \right) \cdot \frac{x-y}{\Y{x-y}}}^2.
\end{align*}
Here, the relative displacement of the fiber endpoints $u(x)-u(y)$ is projected onto the unit vector $\frac{x-y}{\Y{x-y}}$ pointing along the fiber direction. This operation isolates the magnitude of stretch or compression along the fiber direction while ignoring all other motions. 
Consequently, the energy formulation intentionally discards all components of relative displacement that are perpendicular to the fiber axis, as they are assumed not to contribute to the stored energy. This is a common and effective modeling choice in the mechanics of fiber-reinforced materials to represent the highly anisotropic nature of the reinforcement.

In this paper, we consider more general energy functionals with $p$-growth, which enable the modeling of a broader range of material classes.

The free energy of the random long-range connections is given by 
\begin{align}\label{eq:1st_version_fiber_energy}
\Jepsut=    \varepsilon^{2d}
\sum_{x \in Q_\varepsilon} \sum_{y \in Q_\varepsilon} 
\varepsilon^{-d-ps} \ocweight{\left(\tfrac{x}{\varepsilon}, {\tfrac{y}{\varepsilon}}\right)} \Y{\left( u(x)-u(y) \right) \cdot \frac{x-y}{\Y{x-y}}}^p.
\end{align}
The two scaling factors in the above expression for $\Jepsut$ are crucial for the discrete-to-continuum transition. The prefactor $\varepsilon^{2d}$ serves as the volume measure for the double summation over the lattice $Q_\varepsilon$. This ensures that the sum correctly scales to a double integral over the domain $Q \times Q$ in the limit $\varepsilon \rightarrow 0$, in direct analogy to a multidimensional Riemann sum.

The second factor $\varepsilon^{-d-ps}$ is a more subtle probabilistic scaling. It is calibrated to precisely balance the decay rate of the connection probability, which can be written as $\mathbb{P}\left( \ocweight (\frac{x}{\varepsilon},\frac{y}{\varepsilon}) =1 \right) 
= \left(\frac{\Y{x-y}}{\varepsilon}\right)^{-d-ps}$. This choice is essential as it ensures that the expected value of the energy functional remains bounded and non-trivial, since $d>ps$. Consequently, this balance guarantees the convergence to a non-trivial, deterministic homogenized functional, preventing the total energy of the fibers from either vanishing or diverging in the limit.
A schematic overview of the discrete lattice and randomized fiber network is provided in \Cref{fig:fig}. 
While for visualization purposes the fibers are drawn as rigid beams connecting two nodes, note that in the mathematical model the fibers are represented as conductances between pairs of points in the domain. The spheres at the ends of the beams indicate the lattice nodes that serve as endpoints of these fiber interactions.
\begin{figure}[!htb] 
    \centering
    \includegraphics[width=\linewidth]{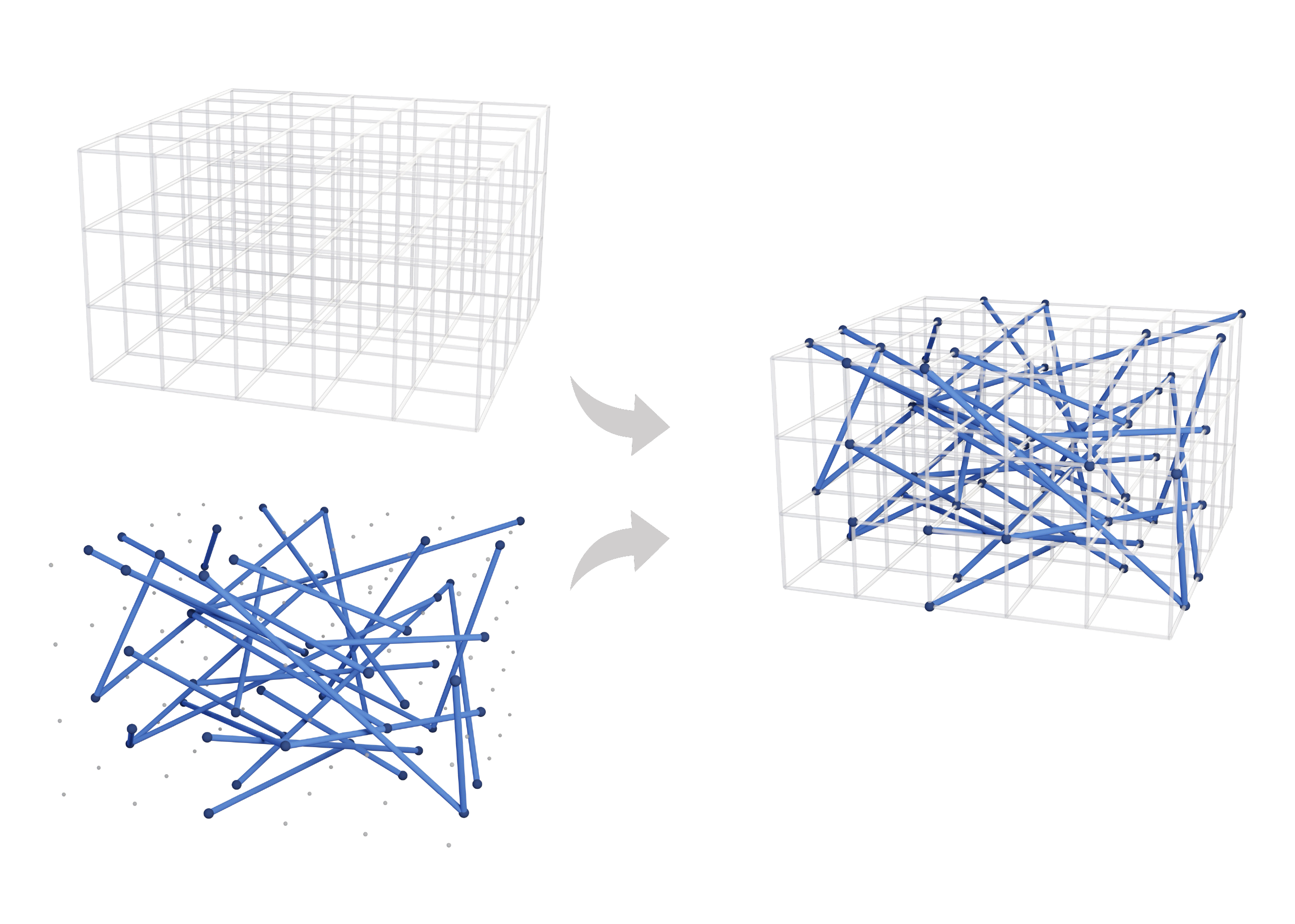} 
    \caption{Schematic representation of the discrete lattice model. 
\textit{Top left}: Nearest-neighbor interactions forming the elastic matrix.
\textit{Bottom left}: Example realization of long-range conductances representing random fibers.
\textit{Right}: Composite network combining matrix springs and fiber reinforcements in the discrete setting.
    }
    \label{fig:fig}
\end{figure}

\noindent With $\coeff \coloneqq  \ocweight{\left(\tfrac{x}{\varepsilon},\frac{y}{\varepsilon}\right)} \left(\tfrac{\Y{x-y}}{\varepsilon}\right)^{d+ps} $, the expression in \eqref{eq:1st_version_fiber_energy} can be reformulated as
\begin{align}
    \Jepsut &= \varepsilon^{2d}
\sum_{x \in Q_\varepsilon} \sum_{y \in Q_\varepsilon} \coeff \frac{\Y{\left( u(x)-u(y) \right) \cdot \tfrac{x-y}{\Y{x-y}}}^p}{\Y{x-y}^{d+ps}} . 
\end{align}
The overall energy of the material system is then
\begin{align}\label{eq:def discrete energy}
\Eeps (t, u_\varepsilon, z_\varepsilon ) = \Jepsut + \Eloct - \mathcal{F}_\varepsilon (t, u_\varepsilon) ,
\end{align}
where 
\begin{align*}
    \mathcal{F}_\varepsilon (t, u_\varepsilon) &= \varepsilon^d \sum_{x\in Q_\varepsilon} f_\varepsilon(t,x) u_\varepsilon(x).
\end{align*}
Furthermore, we define the discrete dissipation functional
\begin{align}\label{eq:def discrete dissipation}
\Psi_\varepsilon (u_\varepsilon , z_\varepsilon ) \coloneqq \varepsilon^d \sum_{x\in \Z_\varepsilon^d} \rho (z_\varepsilon (x)),
\end{align}
where $\rho : \R^{d \times d} \rightarrow [0, \infty )$ is measurable, convex and positively homogeneous of degree 1, i.e. $\rho (a  v) = a \rho (v)$ for any $a \in \R_{\geq 0}$ and any $v\in \R^{d \times d}$.

The corresponding limit functional is given by
\begin{align} \label{eq:limit_func}
\mathcal{E}(t,u,z) \coloneqq \mathcal{E}^{V}(u)  + \mathcal{E}^{loc}(u,z) - \mathcal{F}(t,u),
\end{align}
where
\begin{align*}
\mathcal{E}^{V}(u) & \coloneqq 
\int_Q \int_Q \frac{\Y{\left( u(x)-u(y) \right) \cdot \tfrac{x-y}{\Y{x-y}}}^p}{\Y{x-y}^{d+ps}} \dif x \dif y , \\
\mathcal{E}^{loc}(u,z) & \coloneqq \int_Q e^\mathrm{el}(x) : \bold{A}(x) : e^\mathrm{el}(x) + z(x) : \bold{H}(x) : z(x) + \kappa \Y{\nabla  z (x)}^2 \dif x , \\
\mathcal{F}(t,u) & \coloneqq \int_{Q} f(t,x) u(x) \dif x .
\end{align*}
Here, the elastic part is defined analogously to the discrete setting as $ e^\mathrm{el}(x) \coloneqq \nabla ^s u(x) -z(x)$. The dissipation functional associated to the limit problem is
\begin{align}\label{eq:def dissipation functional}
\Psi (u,z) = \int_{\R^d} \rho (z(x))  \dif x.
\end{align}

Throughout the paper we use the convention that $C$ denotes a generic positive constant that may change from line to line.

Unless stated otherwise, all assumptions introduced in this section are imposed globally and are sufficient for all subsequent results.

\subsection{Operator Framework for Discrete-to-Continuum transition}
In this section, we introduce the tools required to pass from discrete lattice functions to continuum fields and vice versa, summarizing their properties and convergence behavior.
To establish a rigorous foundation for discussing the convergence properties of discrete functions $u_\varepsilon : Q_\varepsilon \rightarrow \mathbb{R}^d$ toward their continuum counterparts $u:Q \rightarrow \mathbb{R}^d$, we introduce the piece-wise constant reconstruction operator $\Reps$. 
This operator transforms discrete functions $u_\varepsilon : Q_\varepsilon \rightarrow \mathbb{R}^d$ into piecewise constant functions $\Repsu : Q \rightarrow \mathbb{R}^d$ through the following mapping
\begin{align*}
\Repsu (x) := 
u_\varepsilon (z) \qquad\text{if } z\in \mathbb{Z}^d_\varepsilon \text{ and } x \in \left( z+\left(-\tfrac{\varepsilon}{2} , \tfrac{\varepsilon}{2}  \right]^d \right)\cap Q .
\end{align*}
In simple words, the operator $\Reps$ assigns to each point within the $\varepsilon$-cell $\left(z+\left(-\tfrac{\varepsilon}{2} , \tfrac{\varepsilon}{2}  \right]^d \right) \cap Q$ the function value of $u_\varepsilon$ evaluated at the cells grid center $z\in \mathbb{Z}^d_\varepsilon$.
This enables us to work with discrete functions in the continuum setting. For the other way around its adjoint operator $\mathcal{R}_\varepsilon$ can be used. It is a discretization operator that transforms functions $u : Q \rightarrow \mathbb{R}^d$ into discrete functions $\mathcal{R}_\varepsilon u : Q_\varepsilon \rightarrow \mathbb{R}^d $ through 
\begin{align*}
    \mathcal{R}_\varepsilon u(x) := \dashint_{\left( x+\left(-\tfrac{\varepsilon}{2} , \tfrac{\varepsilon}{2}  \right]^d \right)\cap Q } u(y) \dif y .
\end{align*}
Therefore, it operates by assigning each point $x \in Q_\varepsilon$ the mean value of $u$ on the entire cube $\left(x+\left(-\tfrac{\varepsilon}{2} , \tfrac{\varepsilon}{2}  \right]^d \right)  \cap Q$.
This operator framework allows us to establish the convergence of $u_\varepsilon$ to $u$ by showing that $\Repsu \rightarrow u$ in the appropriate function space.
Furthermore, with the help of the operator $\Reps$ we can transform discrete summations into continuous integrals through the identity
\begin{align}\label{eq:integral_identity}
\varepsilon^d \sum_{x \in Q_\varepsilon} u_\varepsilon (x) = \int_Q \Repsu (x)  \dif x
\end{align}

\begin{figure}[!htb] 
    \centering
    \includegraphics[width=\linewidth]{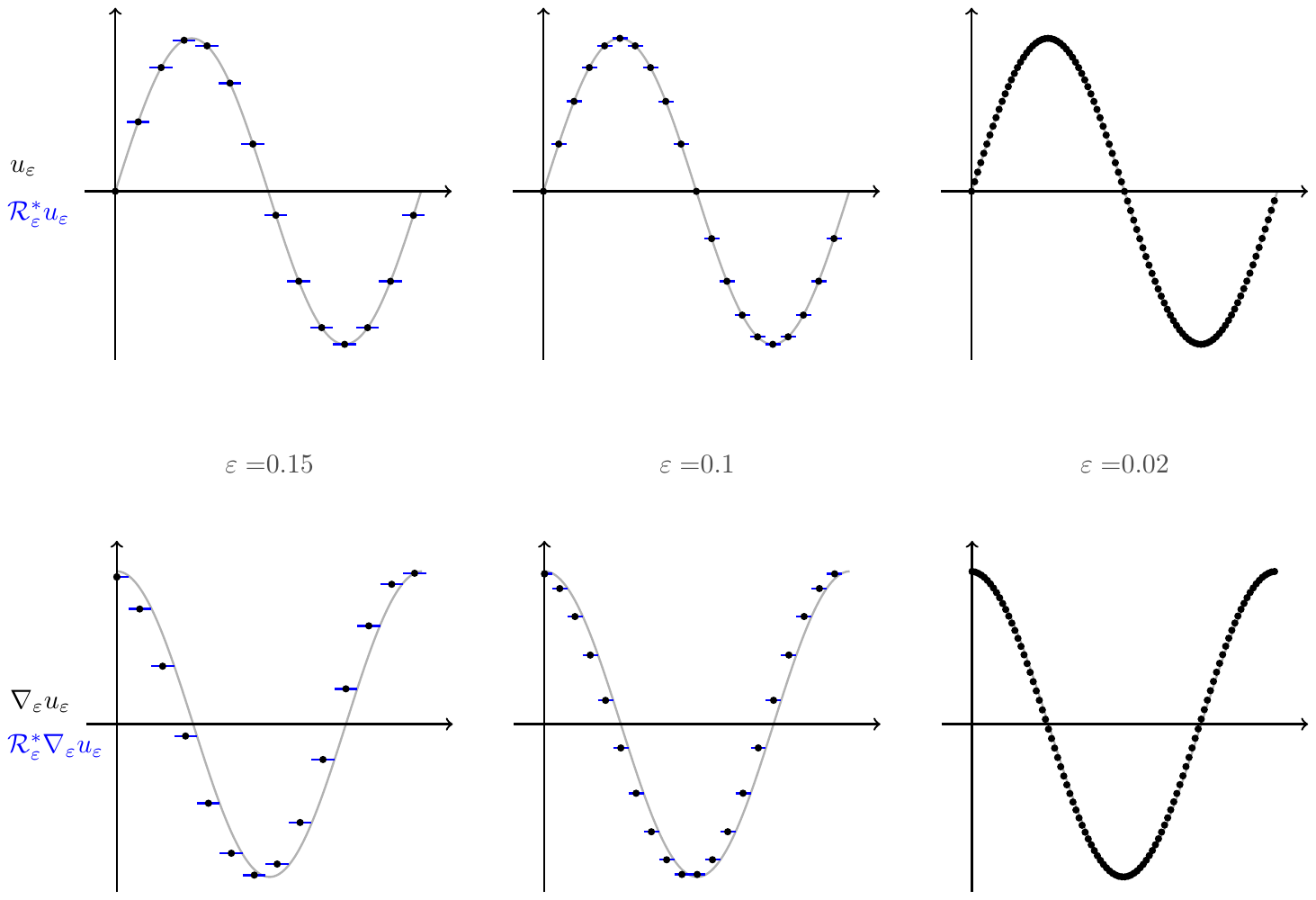} 
    \caption{One-dimensional example for the piecewise reconstruction of a discrete function and its discrete gradient for $\varepsilon \in \{0.15,0.1,0.05 \} $.
    \textit{Top }: The function $u(x)=\sin{(\pi x)}$ (\textcolor{black!40}{grey}) on $Q=[0,2]$, the discrete function $u_\varepsilon (x) = \sin{(\pi x)}$ (black dots) on $Q_\varepsilon = \varepsilon \mathbb{Z} \cap [0,2]$ and the corresponding operator $\Repsu$ (\textcolor{blue}{blue}) on $Q=[0,2]$. 
\textit{Bottom }: The function $\nabla u$ (\textcolor{black!40}{grey}), the discrete gradient $\nabla_\varepsilon u_\varepsilon$ (black dots) and its reconstruction $\Reps \nabla_\varepsilon u_\varepsilon$ (\textcolor{blue}{blue}).     }
    \label{fig:Reps}
\end{figure} 
\noindent if the domain satisfies $Q = \bigcup_{x \in Q_\varepsilon} [x-\frac{\varepsilon}{2}, x+\frac{\varepsilon}{2}]^d$. For arbitrary domains $Q$, this condition generally fails. However, since $\lim_{\varepsilon \rightarrow 0} \varepsilon^d \Y{ Q_\varepsilon} = \Y{Q}$, where $\Y{Q_\varepsilon} = \# \{ Q \cap \Z^d_\varepsilon \} $, any discrepancy introduced by applying equation \eqref{eq:integral_identity} vanishes in the limit process. Consequently, we adopt equation \eqref{eq:integral_identity} as a notational convenience without explicitly tracking the associated approximation error.

We now introduce the precise notion of convergence between discrete and continuum functions that we employ in this work.

Convergence in $L^2$ for a sequence of discrete functions $u_\varepsilon$ is understood as strong convergence of their piecewise constant reconstructions:
\begin{align*}
    \Repsu \rightarrow u ~\text{strongly in}~L^2(Q)^d.
\end{align*}
A central difficulty arises when defining convergence in $H^1$, since the reconstructions $\Repsu$ are piecewise constant and thus do not belong to $H^1(Q)^d$, see \Cref{fig:Reps}.
To overcome this, we additionally require convergence of the discrete gradients, specifically that
\begin{align*}
    \Reps \nabla_\varepsilon u_\varepsilon \rightarrow \nabla u ~\text{strongly in}~L^2(Q)^{d\times d}.
\end{align*}
Accordingly, strong convergence in $H^1$ in the discrete-to-continuum sense will always refer to the joint property
\begin{align*}
    \Repsu \rightarrow u ~\text{strongly in}~L^2(Q)^d \qquad \text{and} \qquad \Reps \nabla_\varepsilon u_\varepsilon \rightarrow \nabla u ~\text{strongly in}~L^2(Q)^{d\times d}.
\end{align*}
Furthermore, all compactness results of the continuum setting can be translated to the discrete setting by making use of the piecewise-constant reconstruction operator $\Reps$.

In the next section, we extend our analysis to rate-independent systems.
\section{Evolution of rate-independent system with plasticity}
The discrete conductance model in \cite{HeidaHermann}, while capturing essential elastic behavior, represents only equilibrium configurations of the system. To develop a comprehensive understanding of material behavior, we must consider how these systems evolve through time as external conditions change.
This evolutionary extension allows us to model realistic material behavior where plastic deformation accumulates irreversibly under varying loads. The rate-independence captures the essential physics that doubling the loading speed doubles the system response speed, while the underlying material relationships remain unchanged. This property makes the mathematical framework particularly suitable for quasi-static processes in elastoplasticity, where inertial effects are negligible compared to the relevant loading timescales.

First we give a short introduction in the general theory and concepts of rate-independent systems that is adapted by \cite{MIELKE2, mielkeLipschitz, mielkeRoubicek, neukamm2018stochastic}. 

\subsection{Governing equations of rate-independent systems} 
For a given initial value $y_0$ in a Banach space $Y$, we aim to find a solution of the rate-independent system associated to the energy functional $\mathcal{E}: I \times Y \rightarrow \R $ and the 1-homogeneous dissipation functional $\Psi : Y \times Y \rightarrow [0, \infty )$ that is described by the following evolutionary variational inequality
\begin{align}
\langle \text{D}\mathcal{E}\left(t,y(t)\right) , v- \dot{y}(t) \rangle + \Psi (v) - \Psi (\dot{y}(t)) \geq 0  \qquad \forall v \in Y \tag{VI}. \label{variationalIeuality}
\end{align}
In the case of convex energy functionals $\mathcal{E}$ the given variational inequality is equivalent to the energetic formulation:\\
\begin{center}
\begin{minipage}{0.8\linewidth}
Find a function $y:I \rightarrow Y$ such that $y_0=y(0)$ that satisfies the following stability and energy balance conditions 
\begin{align}
& \mathcal{E}(t,y(t)) \leq \mathcal{E}(t,\tilde{y}) + \Psi(\tilde{y}-y(t)) \qquad \text{for all } \tilde{y} \in Y \tag{S} \label{stability} \\
&\mathcal{E}(t,y(t)) + \int_0^t \Psi(\dot{y}(s) ) \dif s = \mathcal{E}(0,y(0)) + \int_0^t \partial_s \mathcal{E}(s,y(s)) \dif s \tag{E} \label{energyBalance}
\end{align}
for all $t\in I=[0,T]$. \\
\end{minipage}	
\end{center}
Equivalently to \eqref{stability} one can say that $y(t) \in \mathcal{S}(t)$ for all $t\in I$, where the set of stable states $\mathcal{S}(t)$ is defined as
\begin{align*}
\mathcal{S}(t)\coloneqq \left\{ y\in Y ~\vert ~ \mathcal{E}(t,y) < \infty ~ \text{and}~ \mathcal{E}(t,y) \leq \mathcal{E} (t, \tilde{y}) + \Psi(\tilde{y}-y) ~\text{for all}~ \tilde{y} \in Y  \right\}.
\end{align*}
We need the following notions to state the subsequent existence result.

\begin{definition}[Stable sequence, cf. \cite{mielkeRoubicek} Sec. 2.1.3]
A sequence $(t_n, y_n)_{n \in \mathbb{N}}$ is called a \textbf{stable sequence} of the rate-independent system associated to the energy functional $\mathcal{E}$ and the dissipation functional $\Psi$ if $\sup_{n \in \mathbb{N}} \mathcal{E}(t_n, y_n) < \infty$ and $y_n \in \mathcal{S}(t_n)$ for all $n \in \mathbb{N}$.
\end{definition} 

\begin{definition}[Energetic solution, see \cite{mielkeRoubicek} Definition 2.1.2.]
A function $y : I \rightarrow Y $ is called an \textbf{energetic solution} of the rate-independent system associated to the energy functional $\mathcal{E}$ and the dissipation functional $\Psi$  if $t \mapsto \partial_t \mathcal{E}(t, y(t))$ is integrable and if the stability \eqref{stability} and the energy balance \eqref{energyBalance} hold for all $t \in I$.
\end{definition}

\subsection{Existence of solutions}
There are various existence results in literature for different kind of rate-independent systems. We cite a rather general existence result that is presented in section 2.1 of \cite{mielkeRoubicek}. 
They make the following assumptions:
\begin{align}
&\textit{Dissipation distance:} ~\tag{I} \label{assumptionI}   \\
&\Psi: Y \rightarrow [0,\infty] ~\text{is 
1-homogeneous, lower semi-continuous and satisfies} \nonumber \\
&\qquad (a)~ \forall y_1, y_2 \in Y:  \Psi(y_1-y_2) = 0 \Leftrightarrow y_1 =y_2, \nonumber\\
&\qquad (b)~  \forall y_1, y_2, y_3 \in Y:  \Psi(y_1-y_3) \leq \Psi(y_1-y_2) + \Psi(y_2-y_3). \nonumber \\
\hspace{4pt} \nonumber \\
&\textit{Compactness of sublevels:}  ~ \tag{II} \label{assumptionII} \\
&\forall t \in I : \mathcal{E}(t, \cdot) : Y \rightarrow [0,\infty] \text{ has compact sublevels}. \nonumber\\
\hspace{4pt} \nonumber \\
 \nonumber \\
&\textit{Energetic control of power:}
\tag{III}\label{assumptionIII} \\  
& \text{Dom} ~\mathcal{E} = I \times \text{Dom} ~ \mathcal{E} (0, \cdot), \nonumber\\
& \exists c_\mathcal{E} \in \mathbb{R} , ~\lambda_\mathcal{E} \in L^1(I), ~N_\mathcal{E} \subset I \text{ with } \mathscr{L}^1(N_\mathcal{E}) =0 \nonumber \\
& \forall y \in \text{Dom} ~\mathcal{E}(0, \cdot ) : \mathcal{E}(\cdot, y) \in W^{1,1}(I), \nonumber \\
& \qquad  \qquad \qquad ~\partial_t \mathcal{E}(t,y) \text{ exists for } t \in I \setminus N_\mathcal{E} \text{ and satisfies } \nonumber \\
& \qquad \qquad \qquad \Y{\partial_t \mathcal{E}(t,y)} \leq \lambda_\mathcal{E} (t) \left( \mathcal{E}(t,y) + c_\mathcal{E} \right). \nonumber \\
\hspace{4pt} \nonumber \\
&\textit{Conditions on convergent stable sequences:}
\tag{IV}\label{assumptionIV} \\
&\forall \text{ stable sequences } (t_n , y_n)_{n \in \mathbb{N}} \text{ with } (t_n , y_n) \xlongrightarrow{I \times Y} (t,y) \text{ we have}: \nonumber\\
&\qquad (a)~ t \in I \setminus N_\mathcal{E} \text{ with } N_\mathcal{E} \text{ from } \eqref{assumptionIII} \Rightarrow \partial_t \mathcal{E} (t,y) = \lim_{n \rightarrow \infty } \partial_t \mathcal{E}(t,y_n), \nonumber \\
&\qquad (b)~ y \in \mathcal{S}(t). \nonumber
\end{align}

\begin{theorem}[Existence result, see \cite{mielkeRoubicek} Th. 2.1.6] \label{th:existence of solutions}
Assume that $\mathcal{E}$ and $\Psi$ satisfy assumptions \eqref{assumptionI} -- \eqref{assumptionIV} and assume that the topology of Y restricted to compact sets is separable and metrizable.
Then for each $y_0 \in \mathcal{S}(0)$, there exists an energetic solution $y : I \rightarrow Y$ with $y(0)=y_0$ of the rate-independent system associated to $\mathcal{E}$ and $\Psi$.
\end{theorem}

The metrizability and separability condition in \Cref{th:existence of solutions} is automatically fulfilled for separable Banach spaces. Since we work exclusively with separable Banach spaces throughout this paper, this technical assumption is satisfied without further verification.

\subsection{The rate-independent model}
In the following, we want to consider both the rate-independent systems associated with the discrete functionals $\mathcal{E}_\varepsilon$ and $\Psi_\varepsilon$ as defined in \eqref{eq:def discrete energy} and \eqref{eq:def discrete dissipation}, as well as the rate-independent system associated with $\mathcal{E}$ and $\Psi$ as defined in \eqref{eq:limit_func} and \eqref{eq:def dissipation functional}. The aim is to show that solutions of the discrete systems converge in a proper sense to the solution of the continuum system. Therefore, we first need to guarantee that all considered rate-independent systems have solutions and that the solutions of the limit system are unique. We use the same notation as in the previous introduction to rate-independent systems when we consider the continuum system, i.e. the rate-independent system associated to $\mathcal{E}$ and $\Psi$. When we consider the system associated to $\mathcal{E}_\varepsilon$ and $\Psi_\varepsilon$, we add the index $\varepsilon$ to all relevant notations. For example, we define the set of stable states of the discrete system as
\begin{align*}
\mathcal{S}_\varepsilon(t) \coloneqq \left\{ y_\varepsilon\in Y_\varepsilon ~\vert ~ \mathcal{E}_\varepsilon(t,y_\varepsilon) < \infty ~ \text{and}~ \mathcal{E}_\varepsilon(t,y_\varepsilon) \leq \mathcal{E}_\varepsilon (t, \tilde{y}_\varepsilon) + \Psi_\varepsilon (\tilde{y}_\varepsilon-y_\varepsilon) ~\text{for all}~ \tilde{y}_\varepsilon \in Y_\varepsilon  \right\}.
\end{align*}
The underlying spaces for the analysis of the discrete rate-independent systems are
\begin{align*}
    Y_\varepsilon = U_\varepsilon \times Z_\varepsilon ,
\end{align*}
where
\begin{align*}
    U_\varepsilon \coloneqq \left\lbrace u_\varepsilon : \mathbb{Z}_\varepsilon^d \rightarrow \mathbb{R}^d ~\bigg\vert ~ u_\varepsilon = 0 ~\text{on}~ \mathbb{Z}_\varepsilon^d \setminus Q_\varepsilon , ~ \varepsilon^d \sum_{x \in Q_\varepsilon} \Y{ \nabla_\varepsilon u_\varepsilon (x)}^2 < \infty , ~\text{and}~ \Jepsut < \infty \right\rbrace
\end{align*}
and
\begin{align*}
    Z_\varepsilon \coloneqq \left\lbrace z_\varepsilon: Q_\varepsilon \rightarrow \mathbb{R}_{\text{sym}, 0} ^{d\times d} ~ \bigg \vert ~ \varepsilon^d \sum_{x \in Q_\varepsilon} \Y{ \nabla_\varepsilon z_\varepsilon (x)}^2 < \infty \right\rbrace 
\end{align*}
equipped with the norm

\begin{align*}
    \y{y_\varepsilon}_{Y_\varepsilon}^2 \coloneqq \Jepsut^{\frac{2}{p}} +\varepsilon^d \sum_{x \in Q_\varepsilon} \Y{ \nabla_\varepsilon u_\varepsilon (x)}^2 +  \varepsilon^d \sum_{x \in Q_\varepsilon} \left( \Y{  z_\varepsilon (x)}^2 +\Y{ \nabla_\varepsilon z_\varepsilon (x)}^2 \right) , 
\end{align*}
where $y_\varepsilon=(u_\varepsilon,z_\varepsilon) \in Y_\varepsilon$.
For the continuum system we have 
\begin{align*}
    Y \coloneqq \left( W^{s,p}(Q)^d \cap  H_0^{1}(Q)^d \right) \times H^{1}(Q, \R_{\text{sym}, 0} ^{d\times d})
\end{align*}
with
\begin{align*}
    \y{y}_Y^2 \coloneqq [u]_{W^{s,p}}^2 + \y{u}_{H^1}^2 + \y{z}_{H^1}^2, \qquad \text{where } y=(u,z).
\end{align*}
Following the simplified notation used in the definition of the norm above, we often omit explicit references to the domain $Q$ and the dimensions of the respective function spaces.
Note that by \eqref{eq:integral_identity} the discrete norm can be translated to the continuum framework via $\y{y_\varepsilon}_{Y_\varepsilon}^2  = \mathcal{E}^V (\Repsu)^\frac{2}{p} + \y{\Reps \nabla_\varepsilon u_\varepsilon}_{L^2}^2 + \y{\Reps z_\varepsilon}_{L^2}^2 + \y{\Reps \nabla_\varepsilon z_\varepsilon}_{L^2}^2$.

\section{Statement of the main results}
The following theorems establish a complete mathematical framework for rate-independent elastoplastic systems with random fiber reinforcement. We first prove existence of solutions for both discrete and continuous systems (Theorems \ref{th:existence of discr solutions} and \ref{th:existence_of_solutions_continuum_system}), then establish uniqueness in the continuum limit (\Cref{lemma:uniqueness_of_solutions}), and finally show that discrete solutions converge to the unique continuum solution (\Cref{thm:rate}). Together, these results provide rigorous mathematical foundations for understanding the macroscopic time-dependent behavior of fiber-reinforced materials. Furthermore, this establishes a consistent framework for developing convergent numerical methods.

\begin{theorem}[Existence of solutions to discrete systems] \label{th:existence of discr solutions} 
Let $f_{\varepsilon} \in C^{Lip}(I,L^2(Q_\varepsilon)^d)$.
Then for each initial value $y^0_\varepsilon=(u^0_\varepsilon,z^0_\varepsilon) \in \mathcal{S}_\varepsilon(0)$, there exists an energetic solution $y_\varepsilon =(u_\varepsilon,z_\varepsilon): I \rightarrow Y_\varepsilon$ with $y_\varepsilon(0)=y^0_\varepsilon$ of the rate-independent system associated to $\mathcal{E}_\varepsilon$ and $\Psi_\varepsilon$ as defined in \eqref{eq:def discrete energy} and \eqref{eq:def discrete dissipation}.
\end{theorem}

\begin{theorem}[Existence of solutions to the continuum system] \label{th:existence_of_solutions_continuum_system}
Let $f \in C^{Lip}(I,L^2(Q)^d)$.
Then for each initial value $y_0=(u_0,z_0) \in \mathcal{S}(0)$, there exists an energetic solution $y =(u,z): I \rightarrow Y$ with $y(0)=y^0$ of the rate-independent system associated to $\mathcal{E}$ and $\Psi$ as defined in \eqref{eq:limit_func} and \eqref{eq:def dissipation functional}.
\end{theorem}

~

\noindent The discrete energy functional $\mathcal{E}_\varepsilon$ can be represented in integral form as

\begin{align} \label{eq:integral_representation_discrete_energy}
\Eepsut =& \int_Q \int_Q \Reps \coeff \frac{\Y{\left(\Repsu (x) - \Repsu (y)\right) \cdot \frac{\Reps x -\Reps y}{\Y{\Reps x - \Reps y}}}^p}{\Y{\Reps x- \Reps y}^{d+ps}} \dif x \dif y \\
&+ \int_Q \Big(\left( \nabla^s_\varepsilon \Repsu (x) - \Reps z_\varepsilon (x) \right) : \bold{A}(x) : \left( \nabla^s_\varepsilon \Repsu (x) - \Reps z_\varepsilon (x) \right) \nonumber\\
& ~~~~~~~~~ + \Reps z_\varepsilon (x) : \bold{H}(x) : \Reps z_\varepsilon (x) + \kappa \Y{\nabla_\varepsilon \Reps z_\varepsilon (x)}^2 \Big) \dif x \nonumber \\
&- \int_Q \Reps f_\varepsilon (t,x) \Repsu  \dif x , \nonumber
\end{align}
where
\begin{align*}
\Reps \coeff \coloneqq  c \ocweight{\left( \tfrac{\Reps x}{\varepsilon}, \tfrac{\Reps x}{\varepsilon}- \tfrac{\Reps y}{\varepsilon}\right)} \left(\frac{\Y{\Reps x- \Reps y}}{\varepsilon}\right)^{d+ps} .
\end{align*}

~

\noindent This representation is very similar to that of the limit functional $\mathcal{E}$ that was defined in \eqref{eq:limit_func}.  
Therefore, we only present the proof of existence of solutions to the continuum system (\Cref{th:existence_of_solutions_continuum_system}), because all steps can be translated to the discrete energy $\mathcal{E}_\varepsilon$ via \eqref{eq:integral_representation_discrete_energy}. The existence proof relies on the general existence theory for rate-independent systems (\Cref{th:existence of solutions}) and consists of checking that our energy and dissipation functionals satisfy the necessary structural hypotheses.

\begin{theorem}[Uniqueness of solutions to the continuum system]\label{lemma:uniqueness_of_solutions}
    There is at most one energetic solution of the rate-independent system associated with $\mathcal{E}$ and $\Psi$.
\end{theorem}

\begin{theorem}[Convergence of solutions]\label{thm:rate}

Let $f_{\varepsilon} \in C^{Lip}(I,L^2(Q_\varepsilon)^d)$ be such that $\Reps f_\varepsilon \rightarrow f$ in $L^{2}(Q)^d$,
where $f \in C^{Lip} (I, L^2(Q)^d) $. Let $(u_\varepsilon^0, z_\varepsilon^0) \in \mathcal{S}_\varepsilon(0)$ with
\begin{align}\label{eq:thm:rate:conv1}
\Reps u_\varepsilon^0 \rightarrow u_0 ~\text{in}~ L^2(Q)^d,& ~~~ \Reps \nabla_\varepsilon u_\varepsilon^0 \rightarrow \nabla u_0 ~\text{in}~ L^2(Q)^{d \times d} \nonumber\\ ~~~ &\text{and} ~~~ \\
\Reps z_\varepsilon^0 \rightarrow z_0 ~ \text{in}~ L^2(Q)^{d \times d},& ~~~ \nabla_\varepsilon z_\varepsilon^0 \rightarrow \nabla z_0 ~\text{in}~ L^2(Q)^{d \times d \times d}.\nonumber
\end{align}
Let $y_\varepsilon =(u_\varepsilon , z_\varepsilon) $ be the solution to \eqref{stability} and \eqref{energyBalance} associated to $\mathcal{E}_\varepsilon$ and $\Psi_\varepsilon$ with $u_\varepsilon (0)= u_\varepsilon^0$ and $z_\varepsilon (0) = z_\varepsilon^0 $. Then $(u_0, z_0) \in \mathcal{S}(0)$ and for every $t \in I$: 
\begin{align}\label{eq:thm:rate}
\Reps u_\varepsilon \rightarrow u ~\text{in}~ L^2(Q)^d,& ~~~ \Reps \nabla_\varepsilon u_\varepsilon \rightarrow \nabla u ~\text{in}~ L^2(Q)^{d \times d}, \nonumber\\ 
~~~ &\text{and} ~~~ \\
\Reps z_\varepsilon \rightarrow z ~ \text{in}~ L^2(Q)^{d \times d}, &  ~~~ \Reps \nabla_\varepsilon z_\varepsilon \rightarrow \nabla z ~\text{in}~ L^2(Q)^{d \times d \times d}  .\nonumber
\end{align}
Here $y = (u,z) $ is the solution to \eqref{stability} and \eqref{energyBalance} associated to $\mathcal{E}$ and $\Psi$ with $u (0)= u_0$ and $z(0) = z_0 $.
\end{theorem}

This convergence result provides the rigorous connection between discrete and continuum descriptions. It demonstrates that as the lattice spacing $\varepsilon $ tends to zero, discrete solutions approach the unique continuum solution in the appropriate function spaces, thus validating the homogenization procedure.

The uniqueness result (\Cref{lemma:uniqueness_of_solutions}) is crucial for the convergence in \Cref{thm:rate}. Without uniqueness, discrete solutions could converge to different continuum solutions along different subsequences, preventing 
convergence of the full sequence. The uniqueness ensures that any 
converging subsequence must approach the same limit, allowing us to 
conclude convergence of the entire sequence (see final paragraph in the proof
of \Cref{thm:rate}).

\section{Preliminaries and auxiliary lemmas}
The proofs of our main results rely on different categories of auxiliary results that we collect in this section. First, we establish the equivalence between fractional Sobolev spaces and directional difference quotient spaces (\Cref{sec:frac_spaces}), which allows us to work with the more accessible $W^{s,p}$ framework. The remaining auxiliary lemmas serve distinct purposes in our analysis: Lemmas \ref{m:l:konvJeps}--\ref{thm:Korn} provide compactness and convergence tools for the discrete-to-continuum transition, which are needed to verify the convergences in \Cref{thm:rate}. Lemmas \ref{th:Korn_plasticity}--\ref{lemma:boundedness of norms} establish energy estimates and coercivity bounds, that are essential for Theorems \ref{th:existence of discr solutions} and \ref{th:existence_of_solutions_continuum_system}. Lemmas \ref{lemma:beta}--\ref{lemma:for_uniqness} control nonlinear terms and evolutionary inequalities, as they are crucial for the uniqueness proof in \Cref{lemma:uniqueness_of_solutions}.

\subsection{Fractional spaces of directional difference quotients }\label{sec:frac_spaces}
Following the definitions in \cite{ScottMengesha}, we introduce the fractional spaces of directional differences as
\begin{align*}
    \mathcal{X}_p^s (Q)^d  \coloneqq \left\lbrace u \in L^p(Q)^d ~ \bigg\vert ~ [u]_{\mathcal{X}_p^s}^p \coloneqq \int_Q \int_Q \frac{1}{\Y{x-y}^{d+ps}} \Y{\frac{(u(x)-u(y))\cdot (x-y)}{\Y{x-y}}}^p < \infty \right\rbrace .
\end{align*}
These spaces arise naturally from the non-local part of the energy $\mathcal{E}^V(u)$ since by definition $[u]_{\mathcal{X}_p^s}^p = \mathcal{E}^V(u)$. In \cite{ScottMengesha} the authors showed that these spaces are equivalent to fractional Sobolev spaces in the sense that for any $s\in (0,1)$ and any $p \in (1, \infty) $ it holds $W^{s,p}(\mathbb{R}^d)^d = \mathcal{X}_p^s(\mathbb{R}^d)^d$. Furthermore they provide the following Korn-type inequality.
\begin{lemma}[Fractional Korn's inequality, see \cite{ScottMengesha} Th.1.1]
    For any $s\in (0,1)$ and any $p \in (1, \infty) $ there exists a constant $C = C(d,p,s)$ such that for all $u \in \mathcal{X}_p^s(\mathbb{R}^d)^d$ 
    \begin{align}\label{eq:fractional Korn inequality}
        [u]_{\mathcal{X}_p^s} \leq [u]_{W^{s,p}} \leq C [u]_{\mathcal{X}_p^s} .
    \end{align}
\end{lemma}
This equivalence allows us to use the more accessible and thoroughly studied space $W^{s,p}$ for the definition of $Y$.

\subsection{Auxiliary lemmas}

\begin{lemma}[Convergence of the non-local energy, see \cite{HeidaHermann} Th. 4.1]\label{m:l:konvJeps}
Let $u_\varepsilon : Q_\varepsilon \rightarrow \R^d$ be such that $\Repsu \rightarrow u$ strongly in $L^p(Q)^d$, 
then
\begin{align*}
\liminf_{\varepsilon \rightarrow 0} \Jepsut  \geq \mathcal{E}^V (u).
\end{align*}
If additionally  
\begin{align*}
     \sup_\varepsilon \sup_{x,y \in Q_\varepsilon} \frac{1}{\Y{\Reps x-\Reps y}^{d+ps}} \Y{\left(\Repsu (x)-\Repsu(y) \right) \cdot \tfrac{\Reps x-\Reps y}{\Y{\Reps x- \Reps y}}}^p < \infty  ,
\end{align*}
then
\begin{align*}
\Jepsut \asarrow  \mathcal{E}^V (u).
\end{align*}
\end{lemma}

This lemma provides the key convergence property for the non-local energy term $E^V_\eps$ and is used in the proof of \Cref{thm:rate} to establish the $\liminf$-inequality \eqref{eq:liminf_nonloc} and almost sure convergence.

\begin{lemma}[Discrete Poincar\'e's inequality and compactness, see \cite{HeidaHermann} Th. 3.2] \label{thm:Poincare} 
    There exists a constant $C<\infty$ such that for every $\eps>0$, every $p\in[1,\frac{2d}{d-2})$ if $d>2$ or $p\in[1,\infty)$ if $d=2$ and for every  $u_\eps:\, \Z^d_\eps\to \mathbb{R}^d$ with $u_\eps(x)=0$ for $x\not\in Q_\varepsilon$ it holds
    \begin{equation}\label{eq:Korn2}
    \left(\eps^d\sum_{x\in Q_\eps}|u_\eps(x)|^p\right)^\frac{1}{p}\leq
    C \left(\eps^d\sum_{x\in \Z^d_\eps}
    \Y{\nabla_\varepsilon u_\varepsilon(x)}^2 \right)^\frac{1}{2} .     
    \end{equation}
    Furthermore, for every sequence $u_\eps$ with
    $$\sup_\eps\left(\eps^d\sum_{x\in \Z^d_\eps} \Y{\nabla_\varepsilon u_\varepsilon(x)}^2 \right)^\frac{1}{2} <\infty$$ there exists a subsequence $u_{\eps'}$ and $u\in L^p(Q)^d$ such that $\mathcal{R}_{\eps'}^\ast u_{\eps'}\to u$ strongly in $L^p(Q)^d$ as $\eps'\to0$.
\end{lemma}

\begin{lemma}[Discrete Korn's inequality]\label{thm:Korn}
    There exists a constant $C<\infty$ such that for every  $u_\eps:\, \Z^d_\eps\to \mathbb{R}^d$ with $u_\varepsilon = 0$ on $\Z_\varepsilon^d\setminus Q_\varepsilon$ 
    
    it holds
    \begin{equation}\label{eq:Korn}
    \sum_{x\in Q_\eps} \Y{\nabla_\varepsilon u_\varepsilon(x)}^2 \leq C \sum_{x\in Q_\eps} \Y{\nabla^s_\varepsilon u_\varepsilon(x)}^2 .        
    \end{equation}
\end{lemma}
A proof of the above discrete Korn's inequality can be found in the second part of the proof of Theorem 3.1 in \cite{HeidaHermann}.

\begin{lemma}{(Discrete Korn's inequality with plasticity)}\label{th:Korn_plasticity}  \\
There exists a constant $C < \infty $ such that for every $(u_\varepsilon , z_\varepsilon) : \Z_\varepsilon ^d \times Q_\varepsilon^d \rightarrow \R^d \times \R^{d \times d}$ with $u_\varepsilon = 0 $ on $ \Z_\varepsilon ^d \setminus Q_\varepsilon $ it holds
\begin{align*}
\y{\nabla _\varepsilon u_\varepsilon}_{\ell^2(Q_\varepsilon)^{d \times d}}^2 \leq C \mathcal{E}_\varepsilon ^{loc} (u_\varepsilon , z_\varepsilon).
\end{align*}
\end{lemma}

This enhanced Korn inequality with plasticity is essential for establishing the coercivity estimates in the proofs of Theorems \ref{th:existence of discr solutions} and \ref{th:existence_of_solutions_continuum_system}, particularly in the verification of assumption \eqref{assumptionII}. 

\begin{proof}[Proof of \Cref{th:Korn_plasticity}]

\begin{align*}
\mathcal{E}_\varepsilon ^{loc} (u_\varepsilon , z_\varepsilon) 
&\geq \lambda_{\bold{A}}^\mathrm{min} \y{\nabla_\varepsilon ^s u_\varepsilon - z_\varepsilon }_{\ell^2}^2 + \lambda_{\bold{H}}^\mathrm{min} \y{z_\varepsilon}_{\ell^2}^2 + \kappa \y{\nabla_\varepsilon z_\varepsilon}_{\ell^2}^2 \\
&\geq \lambda_{\bold{A}}^\mathrm{min}  ( \y{\nabla_\varepsilon ^s u_\varepsilon - z_\varepsilon }_{\ell^2}^2 + \y{z_\varepsilon}_{\ell^2}^2 ) \\
&\geq \tfrac{1}{2} \lambda_{\bold{A}}^\mathrm{min}  ( \y{\nabla_\varepsilon ^s u_\varepsilon - z_\varepsilon }_{\ell^2} + \y{z_\varepsilon}_{\ell^2} )^2 \\
&\geq \tfrac{1}{2} \lambda_{\bold{A}}^\mathrm{min}   \y{\nabla_\varepsilon ^s u_\varepsilon }_{\ell^2} ^2 \\
&\geq c_\text{Korn} \tfrac{1}{2} \lambda_{\bold{A}}^\mathrm{min}   \y{\nabla_\varepsilon u_\varepsilon }_{\ell^2} ^2,
\end{align*}
where in the last inequality the standard discrete Korn inequality (\Cref{thm:Korn}) was used.
\end{proof}

\begin{lemma}
Let $h : \mathbb{R}^d \rightarrow \mathbb{R}^d$ with $h=0$ on $\mathbb{R}^d\setminus Q$ and $w: Q \rightarrow\mathbb{R}^{d \times d}$, then there exists a $\estconst >0$ such that
\begin{align}\label{eq:ineq1}
    \lambda_{\bold{A}}^{min}\y{\nabla^s h - w}^2_{L^2} +\lambda_{\bold{H}}^{min} \y{w}_{L^2}^2  \geq \estconst  \left( \y{ h }_{H^1}^2 + \y{ w}_{H^1}^2 \right) - \kappa \y{\nabla w}^2_{L^2}.
\end{align}
\end{lemma}

\begin{proof}
A simple calculation shows
\begin{align*}
\y{\nabla ^s h }^2_{L^2} &\leq \y{\nabla ^s h - w}_{L^2}^2 + \y{w }_{L^2}^2 + 2 \y{\nabla ^s h - w}_{L^2} \y{w }_{L^2} \\
&\leq \left( 1 + \frac{2 \lambda_{\bold{A}}^{min}}{\lambda_{\bold{H}}^{min}}\right) \y{\nabla ^s h - w}_{L^2}^2 + \left( 1 + \frac{\lambda_{\bold{H}}^{min}}{2 \lambda_{\bold{A}}^{min}}\right) \y{w }_{L^2}^2 ,
\end{align*}
where in the last step the $\varepsilon$-Young inequality ($ab \leq \frac{a ^2}{2 \varepsilon}  + \frac{\varepsilon b^2}{2} $) was used with $\varepsilon = \frac{\lambda_{\bold{H}}^{min}}{2 \lambda_{\bold{A}}^{min}}$.
Rearranging the above inequality for $\y{\nabla^s h - w}^2_{L^2}$ gives the estimate
\begin{align*}
    \lambda&_{\bold{A}}^{min}\y{\nabla^s h - w}^2_{L^2} +\lambda_{\bold{H}}^{min} \y{w}_{L^2}^2 \\
    &\geq \frac{\lambda_{\bold{A}}^{min} \lambda_{\bold{H}}^{min} }{\lambda_{\bold{H}}^{min} + 2\lambda_{\bold{A}}^{min} } \left( \y{\nabla^s h}^2_{L^ 2} - \left( 1+ \frac{\lambda_{\bold{H}}^{min}}{2\lambda_{\bold{A}}^{min}}  \right) \y{w}^2_{L^2} \right) +\lambda_{\bold{H}}^{min} \y{w}_{L^2}^2 \\
    &\geq \frac{\lambda_{\bold{A}}^{min} \lambda_{\bold{H}}^{min} c_K }{\lambda_{\bold{H}}^{min} + 2\lambda_{\bold{A}}^{min} }  \y{ h}^2_{H^ 1} + \frac{\lambda_{\bold{H}}^{min}}{2}  \y{w}^2_{L^2} + \kappa \y{\nabla w}^2_{L^2} - \kappa \y{\nabla w}^2_{L^2} \\
    &\geq \frac{\lambda_{\bold{A}}^{min} \lambda_{\bold{H}}^{min} c_K }{\lambda_{\bold{H}}^{min} + 2\lambda_{\bold{A}}^{min} }  \y{ h}^2_{H^ 1} + \min\left\lbrace \frac{\lambda_{\bold{H}}^{min}}{2} , \kappa \right\rbrace \y{ w}^2_{H^1} - \kappa \y{\nabla w}^2_{L^2} \\
    &\geq \estconst  \left( \y{ h }_{H^1}^2 + \y{ w}_{H^1}^2 \right) - \kappa \y{\nabla w}^2_{L^2},
\end{align*}
where $c_K$ is a constant that contains the constant from Korn's inequality and 
\begin{align}\label{eq:defintiona_alpha}
\estconst = \min \left\lbrace \frac{\lambda_{\bold{A}}^{min}\lambda_{\bold{H}}^{min} c_K}{\lambda_{\bold{H}}^{min} + 2\lambda_{\bold{A}}^{min} }, \min\left\lbrace \frac{\lambda_{\bold{H}}^{min}}{2} , \kappa \right\rbrace \right\rbrace > 0 .
\end{align}

\end{proof}

\begin{lemma} \label{lemma:boundedness of norms}
    Let $\eta >0$, and $\mathcal{E}(t,y) \leq \eta$, then
    \begin{align}
          C [u]_{W^{s,p}}^p +\estconst \left( \y{ u}_{H^1}^2 + \y{ z}_{H^1}^2 \right)- C_f \y{u}_{L^2} \leq \eta ,
    \end{align}
    where $C_f \coloneqq \y{f(t,\cdot)}_{L^2(Q)^d} < \infty$ and $\estconst$ as defined in \eqref{eq:defintiona_alpha}.
    As a consequence 
    \begin{align}
        [u]_{W^{s,p}}, \y{ u}_{H^1},  \y{ z}_{H^1}  < \infty.
    \end{align}
\end{lemma}
\begin{proof}
    \begin{align}
\eta \geq& ~ \mathcal{E}(t,y)  \nonumber\\ 
=& ~ 
\int_Q \int_Q \frac{1}{\Y{x-y}^{d+ps}} \Y{\left( u(x)-u(y) \right) \cdot \tfrac{x-y}{\Y{x-y}}}^p \dif x \dif y  \nonumber\\
&+ \int_Q \left(\nabla ^s u(x) - z(x) \right) : \bold{A}(x) : \left(\nabla ^s u(x) - z(x) \right) \dif x \nonumber\\
&+ \int_Q z(x) : \bold{H}(x) : z(x) + \kappa \Y{\nabla  z (x)}^2 \dif x \nonumber\\
&- \int_{Q} f(t,x) u(x) \dif x  \nonumber\\
\overset{\substack{\eqref{eq:fractional Korn inequality} \\ \eqref{eq_eigenvalues}}}{~~\geq}& ~  [u]_{\mathcal{X}_p^s}^p  + \lambda_{\bold{A}}^{min} \y{\nabla ^s u - z }_{L^2}^2 + \lambda_{\bold{H}}^{min} \y{z}_{L^2}^2 + \kappa \y{\nabla z}_{L^2}^2 \nonumber \\
&- \y{f(t,\cdot)}_{L^2} \y{u(x)}_{L^2}  \nonumber\\
\overset{\eqref{eq:ineq1}}{~~\geq}& ~  C [u]_{W^{s,p}}^p +\estconst \left( \y{ u }_{H^1}^2 + \y{ z}_{H^1}^2 \right)- C_f \y{u}_{L^2}. \label{eq:lower_estimate_energy} 
\end{align}
Here $C_f \coloneqq \y{f(t,\cdot)}_{L^2(Q)^d} < \infty$ and $\estconst$ as defined in \eqref{eq:defintiona_alpha}.

On the last term on the right-hand side we apply the $\varepsilon$-Young inequality ($ab \leq \frac{a ^2}{2 \varepsilon}  + \frac{\varepsilon b^2}{2} $) with $\varepsilon = \nu$ and get
\begin{align*}
    C_f \y{u}_{L^2} \leq \frac{C_f^2}{2\nu} + \frac{\nu}{2} \y{u}_{L^2}^2 \leq \frac{C_f^2}{2\nu} + \frac{\nu}{2} \y{u}_{H^1}^2.
\end{align*}
Therefore 
\begin{align*}
    \eta \geq  C [u]_{W^{s,p}}^p + \nu \y{ z}_{H^1}^2 +\frac{\nu}{2} \y{ u }_{H^1}^2 - \frac{C_f^2}{2\nu}
\end{align*}
and we can conclude:
\begin{align} \label{eq:bounds_from_E_0}
[u]_{W^{s,p}}, \y{ u }_{H^1},  \y{ z}_{H^1} < \infty.
\end{align}

\end{proof}

~

All remaining lemmas in this section are used to prove \Cref{lemma:uniqueness_of_solutions}, which establishes the uniqueness of solutions for the continuum system.

\begin{lemma}\label{lemma:beta}
    Let $y_\mathrm{min}$ be an energetic solution of the rate-independent system associated to $\mathcal{E}$ and $\Psi$ as defined in \eqref{eq:limit_func} and \eqref{eq:def dissipation functional}. Then there exists a $\beta> 0$ such that the inequality 
    \begin{align} \label{eq:alpha_connvexity}
        \mathcal{E}(t,y) - \mathcal{E}(t,y_\mathrm{min}) + \Psi(y_\mathrm{min} - y) \geq \beta \y{ y_\mathrm{min} - y}_Y
    \end{align}
    is valid for any $y \in Y$ and any $t\in I$. 
\end{lemma}

This strict convexity estimate is essential for the uniqueness proof of \Cref{lemma:uniqueness_of_solutions}, providing the $\beta$-coercivity property needed for the Bihari-type argument.

\begin{proof}[Proof of \Cref{lemma:beta}]
An energetic solution $y_\mathrm{min}$ of the considered rate-independent system satisfies the stability condition \eqref{stability}, i.e. $\mathcal{E}(t,y_\mathrm{min}) \leq \mathcal{E}(t, \tilde{y}) + \Psi(\tilde{y}-y_\mathrm{min})$ for any $\tilde{y} \in Y$.
For any $y\in Y$ we take $\tilde{y} \coloneqq y_\mathrm{min} + \delta (y- y_\mathrm{min}) $ with $\delta >0$. 
Then it is
\begin{align*}
    0 &\leq \mathcal{E}(t, y_\mathrm{min} + \delta (y - y_\mathrm{min})) - \mathcal{E}(t,y_\mathrm{min}) + \Psi(\delta(y-y_\mathrm{min})) \\
    \Rightarrow 0 &\leq \frac{\mathcal{E}(t, y_\mathrm{min} + \delta (y - y_\mathrm{min})) - \mathcal{E}(t,y_\mathrm{min})}{\delta} + \frac{\Psi(\delta(y-y_\mathrm{min}))}{\delta} \\
    & = \frac{\mathcal{E}(t, y_\mathrm{min} + \delta (y - y_\mathrm{min})) - \mathcal{E}(t,y_\mathrm{min})}{\delta} + \Psi(y-y_\mathrm{min}),
\end{align*}
where in the last step, the 1-homogeneity of $\Psi$ was used. Passing to the limit $\delta \rightarrow 0$ in the previous inequality yields 
\begin{align} \label{eq:bew_uniq_1}
    0 \leq \langle D \mathcal{E} (t, y_\mathrm{min}) , y -y_\mathrm{min} \rangle + \Psi(y-y_\mathrm{min}).
\end{align}
In order to verify the desired inequality \eqref{eq:alpha_connvexity}, we write its left-hand side as
\begin{align*}
    \mathcal{E}(t,y) - \mathcal{E}(t,y_\mathrm{min}) + \Psi(y_\mathrm{min} - y) 
    &= \mathcal{E}^V(u) - \mathcal{E}^V(u_\mathrm{min}) + \mathcal{E}^{loc}(u,z) - \mathcal{E}^{loc}(u_\mathrm{min}, z_\mathrm{min}) \\
    &\quad+ \mathcal{F}(t,u_\mathrm{min}) - \mathcal{F}(t,u) + \Psi(y_\mathrm{min} - y)
\end{align*}
and further estimate the respective terms. 
For notational simplicity, we define
\begin{align*}
&K(x,y)\coloneqq \frac{1}{\Y{x-y}^{d+ps}} ,  \\
&w =w(x,y) \coloneqq \left(u(x)-u(y)\right) \cdot \frac{x-y}{\Y{x-y}}, \\
&w_\mathrm{min} = w_\mathrm{min}(x,y) \coloneqq \left(u_\mathrm{min}(x) - u_\mathrm{min}(y)\right) \cdot\frac{x-y}{\Y{x-y}}.
\end{align*}
Then
\begin{align}
    &\mathcal{E}^V(u) - \mathcal{E}^V(u_\mathrm{min}) =  \int_Q \int_Q K(x,y) \left( \Y{w(x,y)}^p - \Y{w_\mathrm{min}(x,y)}^p \right) \dif x \dif y \nonumber \\
    &\overset{\eqref{eq:lindqvist}}{\geq}  \int_Q \int_Q K(x,y) \left( p \Y{w_\mathrm{min}}^{p-2} w_\mathrm{min} \cdot (w - w_\mathrm{min}) + \frac{\Y{w- w_\mathrm{min}}^p}{2^{p-1}-1}
     \right) \dif x \dif y \nonumber \\    
    &\overset{\eqref{eq:fractional Korn inequality}}{\geq} \langle D \mathcal{E}^V(u_\mathrm{min}), u-u_\mathrm{min} \rangle + \frac{1}{2^{p-1}-1} [u-u_\mathrm{min}]_{W^{s,p}}^p. \label{eq:alpha_convexity_eqEV}
\end{align}
To estimate the difference of the $\mathcal{E}^{loc}$-terms we note that
\begin{align}
     \mathcal{E}^{loc}(u,z) 
     =~& \int_Q (e^\mathrm{el}_\mathrm{min} + (e^\mathrm{el}-e^\mathrm{el}_\mathrm{min})) : \bold{A} : (e^\mathrm{el}_\mathrm{min} + (e^\mathrm{el}-e^\mathrm{el}_\mathrm{min})) \nonumber \\
     &\quad+ (z_\mathrm{min} + (z -z_\mathrm{min})) : \bold{H}: (z_\mathrm{min} + (z -z_\mathrm{min})) \nonumber \\
     &\quad+ \kappa \Y{\nabla z_\mathrm{min} + (\nabla z - \nabla z_\mathrm{min})}^2 \dif x  \nonumber\\
     \geq~& \int_Q e^\mathrm{el}_\mathrm{min} :\bold{A}: e^\mathrm{el}_\mathrm{min} + 2 e^\mathrm{el}_\mathrm{min} :\bold{A}: (e^\mathrm{el}-e^\mathrm{el}_\mathrm{min}) + (e^\mathrm{el}-e^\mathrm{el}_\mathrm{min}):\bold{A}:(e^\mathrm{el}-e^\mathrm{el}_\mathrm{min}) \nonumber \\
     &\quad+ z_\mathrm{min}:\bold{H}:z_\mathrm{min} + 2z_\mathrm{min}:\bold{H}:(z-z_\mathrm{min}) + (z-z_\mathrm{min}):\bold{H}:(z-z_\mathrm{min}) \nonumber \\
     &\quad+ \kappa \left( \Y{\nabla z_\mathrm{min}}^2 + 2 \nabla z_\mathrm{min} : (\nabla z -\nabla z_\mathrm{min}) + \Y{\nabla z -\nabla z_\mathrm{min}}^2 \right) \dif x \nonumber \\
     \overset{\eqref{eq_eigenvalues}}{\geq}& \mathcal{E}^{loc}(u_\mathrm{min}, z_\mathrm{min}) + \langle D\mathcal{E}^{loc}(y_\mathrm{min}), y-y_\mathrm{min}\rangle  + \lambda_{\bold{A}}^{min}\y{e^\mathrm{el}-e^\mathrm{el}_\mathrm{min}}^2_{L^2} \nonumber \\
     & + \lambda_{\bold{H}}^{min} \y{z-z_\mathrm{min}}^2_{L^2} + \kappa \y{\nabla z -\nabla z_\mathrm{min}}^2_{L^2} \nonumber \\
     =~& \mathcal{E}^{loc}(u_\mathrm{min}, z_\mathrm{min}) + \langle D\mathcal{E}^{loc}(y_\mathrm{min}), y-y_\mathrm{min}\rangle \nonumber \\
     &+ \lambda_{\bold{A}}^{min} \y{\nabla^s u - \nabla^s u_\mathrm{min} - (z-z_\mathrm{min})}^2_{L^2} + \lambda_{\bold{H}}^{min} \y{z-z_\mathrm{min}}^2_{L^2} \nonumber \\
     &+ \kappa \y{\nabla z -\nabla z_\mathrm{min}}^2_{L^2} \nonumber \\
     \overset{\eqref{eq:ineq1}}{\geq}& \mathcal{E}^{loc}(u_\mathrm{min}, z_\mathrm{min}) + \langle D\mathcal{E}^{loc}(y_\mathrm{min}), y-y_\mathrm{min}\rangle \nonumber \\
     &+ \estconst \left( \y{u- u_\mathrm{min}}_{H^1}^2 + \y{z-z_\mathrm{min}}_{H^1}^2 \right).  \label{eq:alpha_convexity_eqEloc}
\end{align}
Regarding the difference of the $\mathcal{F}$-terms we have
\begin{align}
     \mathcal{F}(t,u_\mathrm{min}) - \mathcal{F}(t,u) = \int_Q f(t,x) ( u_\mathrm{min} - u ) \dif x
     = - \langle D\mathcal{F}(u_\mathrm{min}) , u- u_\mathrm{min} \rangle  .\label{eq:alpha_convexity_eqF}
\end{align}
Combining \eqref{eq:alpha_convexity_eqEV}--\eqref{eq:alpha_convexity_eqF} we end up with
\begin{align*}
    \mathcal{E}(t,y)& - \mathcal{E}(t,y_\mathrm{min}) + \Psi(y_\mathrm{min} - y) \\
    \geq~&  \langle D\mathcal{E}(y_\mathrm{min}), y-y_\mathrm{min}\rangle  + \Psi(y_\mathrm{min} - y)  \\
    &+\frac{1}{2^{p-1}-1} [u-u_\mathrm{min}]_{W^{s,p}}^p + \estconst \left( \y{u- u_\mathrm{min}}_{H^1}^2 + \y{z-z_\mathrm{min}}_{H^1}^2 \right)\\
    \overset{\eqref{eq:bew_uniq_1}}{\geq}& \beta \y{y- y_\mathrm{min}}_Y .
\end{align*}
\end{proof}

\begin{lemma}[Lipschitz-continuity of energetic solutions] \label{lemma:lipschitz_continuity}
    Let $y: I \rightarrow Y$ be an energetic solution of the rate-independent system associated to $\mathcal{E}$ and $\Psi$. Then for any $s, t \in I$ there exists a positive constant $C$ such that the solution $y$ satisfies
    \begin{align} \label{eq:lipschitz_estimate}        
        [u(s)-u(t)]_{W^{s,p}}, \y{u(s)-u(t)}_{H^1}, \y{z(s)-z(t)}_{H^1} \leq C \Y{t-s}.
    \end{align}
\end{lemma}

The proof can be done exactly as in the proof of Theorem 3.4 in \cite{MIELKE2}. For the sake of completeness, the proof is presented in a more detailed version below.
\begin{proof}[Proof of \Cref{lemma:lipschitz_continuity} (Lipschitz-continuity of energetic solutions)]
Inserting $y_\mathrm{min}= y(t)$ and $y = y(s)$ in \eqref{eq:alpha_connvexity}, we get the estimate
\begin{align*}
    \beta \y{y(t)-y(s)}_Y \leq~& \mathcal{E}(t, y(s)) - \mathcal{E}(t, y(t)) + \Psi(y(s)-y(t)) \\
    \leq~& \mathcal{E}(t, y(s)) - \mathcal{E}(t, y(t)) + \int_t^s\Psi(\dot{y}(\tau)) \dif \tau\\
    \overset{\eqref{energyBalance}}{=}& \mathcal{E}(t, y(s))- \mathcal{E}(s, y(s)) + \int_t^s \partial_\tau \mathcal{E}(\tau, y(\tau)) \dif \tau \\
    =~& \int_s^t \partial_\tau \mathcal{E}(\tau, y(s))+ \int_t^s \partial_\tau \mathcal{E}(\tau, y(\tau)) \dif \tau \\
    =~& \int_s^t \partial_\tau \mathcal{E}(\tau, y(s)) - \partial_\tau \mathcal{E}(\tau, y(\tau)) \dif \tau  \\
    \leq~& \int_s^t \y{\dot{f}(\tau)}_{L^2} \y{u(\tau)-u(t)}_{L^2} \dif \tau \\
    \leq~& C_f  \int_s^t \left(\y{y(\tau)-y(t)}_Y \right)^\frac{1}{2}\dif \tau .
\end{align*}
Defining $d(s) \coloneqq  \left(\y{y(s)-y(t)}_Y \right)^\frac{1}{2}$, the above inequality reduces to 
\begin{align}\label{eq:lipsch_eq1}
    d(s)^2 \leq \frac{C_f}{\beta} \int_s^t d(\tau) \dif \tau.
\end{align}
Then we consider $h(r) \coloneqq \int_{t-r}^t d(\tau) \dif \tau$ for $r>0$. According to the Leibniz integral rule, its derivative turns out to be $h^\prime (r) = d(t-r)$. We combine this identity with \eqref{eq:lipsch_eq1} and get
\begin{align}\label{eq:lipsch_eq2}
    h^\prime (r) = d(t-r) \leq \left(\frac{C_f}{\beta} \int_{t-r}^t d(\tau) \dif \tau\right)^\frac{1}{2} = \left(\frac{C_f}{\beta} h(r) \right)^\frac{1}{2}
\end{align}
and thus
\begin{align*}
    2 \frac{\dif}{\dif r}\sqrt{h(r)}= \frac{h^\prime (r)}{\sqrt{h(r)}} \leq \left(\frac{C_f}{\beta} \right)^\frac{1}{2}.
\end{align*}
Integration from $0$ to $r$ and using $h(0)=0$, this turns into
$2 \sqrt{h(r)} \leq \left(\frac{C_f}{\beta} \right)^\frac{1}{2} r$. We rewrite this inequality as $h(r) \leq \frac{C_f}{\beta}\frac{r^2}{4}$ and combine it with \eqref{eq:lipsch_eq2} to deduce
\begin{align*}
    d(t-r) \leq \frac{C_f}{2\beta} r.
\end{align*}
Inserting the explicit definition of $d(t-r)$, we have $\left(\y{y(t-r)-y(t)}_Y \right)^\frac{1}{2} \leq \frac{C_f}{2\beta} r $ and from this we get 
\begin{align*}
   [u(t-r)-u(t)]_{W^{s,p}}, \y{u(t-r)-u(t)}_{H^1}, \y{z(t-r)-z(t)}_{H^1} \leq C r .
\end{align*}
\end{proof}

\begin{lemma}
    Let $y_1, y_2$ be solutions of the rate-independent system associated with $\mathcal{E}$ and $\Psi$, then
    \begin{align}\label{eq:mon1}
        \langle \text{D}\mathcal{E}\left(t,y_1(t)\right) -\text{D}\mathcal{E}(t,y_2(t)), \dot{y}_1(t) - \dot{y}_2(t) \rangle \leq 0.
    \end{align}
\end{lemma}
\begin{proof}
    Since $y_1, y_2$ are energetic solutions and thus satisfy the variational inequality
    \eqref{variationalIeuality}, we chose the respective other solution in \eqref{variationalIeuality} as the test function and get the two inequalities:
    \begin{align}
        \langle \text{D}\mathcal{E}\left(t,y_1(t)\right) ,  \dot{y}_1(t) - \dot{y}_2(t) \rangle + \Psi (\dot{y}_1(t)) - \Psi (\dot{y}_2(t)) &\leq 0 \label{eq:varineq1} \\
        -\langle \text{D}\mathcal{E}\left(t,y_2(t)\right) ,  \dot{y}_1(t) - \dot{y}_2(t) \rangle - \Psi (\dot{y}_1(t)) + \Psi (\dot{y}_2(t)) &\leq 0 \label{eq:varineq2}
    \end{align}
Addition of \eqref{eq:varineq1} and \eqref{eq:varineq2} results in
\begin{align*}
        \langle \text{D}\mathcal{E}\left(t,y_1(t)\right) -\text{D}\mathcal{E}(t,y_2(t)), \dot{y}_1(t) - \dot{y}_2(t) \rangle \leq 0.
    \end{align*}
\end{proof}

\begin{lemma}\label{lemma:for_uniqness}
For any $y_1 , y_2 \in Y$ we have
\begin{align}\label{eq:unif_convexity_est}
    \langle D \mathcal{E}(t,y_1(t)) -& D \mathcal{E}(t,y_2(t)), y_1(t)-y_2(t) \rangle \nonumber \\
    &\geq C \left(  [u_1 -u_2]_{W^{s,p}}^p +  \y{u_1 -u_2}^2_{H^1} + \y{z_1 -z_2}^2_{H^1} \right) .    
\end{align}
\end{lemma}
\begin{proof}
With $w_i \coloneqq \left(u_i(x)-u_i(y)\right) \cdot \frac{x-y}{\Y{x-y}}$ and $K(x,y) \coloneqq \Y{x-y}^{-d-ps}$ we get
    \begin{align*}
        &\langle D \mathcal{E}^V(t,y_1(t)) -D \mathcal{E}^V(t,y_2(t)), y_1(t)-y_2(t) \rangle \\
        &= \bar{c} \int_Q \int_Q K(x,y) \left(\Y{w_1}^{p-2} w_1 - \Y{w_2}^{p-2}w_2 \right)\cdot \left( w_1 - w_2 \right) \dif x \dif y 
    \end{align*}
By \Cref{lemma:lindqvist} we have the inequalities
\begin{align}
    \Y{w_1}^p - \Y{w_2}^p - p\Y{w_2}^{p-2}w_2\cdot(w_1-w_2) &\geq \frac{\Y{w_1 -w_2}^p}{2^{p-1}-1} \label{eq:lindq_1}, \\
    \Y{w_2}^p - \Y{w_1}^p - p\Y{w_1}^{p-2}w_1\cdot(w_2-w_1) &\geq \frac{\Y{w_1 -w_2}^p}{2^{p-1}-1} \label{eq:lindq_2}.
\end{align}
Adding equations \eqref{eq:lindq_1} and \eqref{eq:lindq_2} yields
\begin{align*}
    \left( \Y{w_1}^{p-2} w_1 - \Y{w_2}^{p-2} w_2 \right) (w_1- w_2) \geq \frac{2}{p}\frac{\Y{w_1- w_2}^{p}}{2^{p-1}-1},
\end{align*}
and therefore
\begin{align*}
    \langle D \mathcal{E}^V(t,y_1(t)) -D \mathcal{E}^V(t,y_2(t)), y_1(t)-y_2(t) \rangle 
    &~\geq \frac{2}{p(2^{p-1}-1)} [u_1 -u_2]_{\mathcal{X}_p^s}^p \\
    &\overset{\eqref{eq:fractional Korn inequality}}{\geq} \frac{2C}{p(2^{p-1}-1)} [u_1 -u_2]_{W^{s,p}}^p.
\end{align*}
Furthermore, it is
\begin{align*}
    \langle D \mathcal{E}^{loc}(t,y_1(t)) -& D \mathcal{E}^{loc}(t,y_2(t)), y_1(t)-y_2(t) \rangle \\
    =~& \int_Q (e^\mathrm{el}_1 - e^\mathrm{el}_2) : \bold{A}(x) : (e^\mathrm{el}_1 - e^\mathrm{el}_2) \\
    &~~~~+ (z_1-z_2): \bold{H}(x):(z_1-z_2) + 2 \kappa \nabla(z_1 -z_2): \nabla(z_1-z_2) \dif x \\
    \overset{\eqref{eq_eigenvalues}}{\geq}& \lambda_{\bold{A}}^{min} \y{e^\mathrm{el}_1 - e^\mathrm{el}_2}_{L^2}^2 + \lambda_{\bold{H}}^{min} \y{z_1 -z_2}_{L^2}^2 + 2 \kappa \y{\nabla z_1 -\nabla z_2}_{L^2}^2 \\
    \overset{\eqref{eq:ineq1}}{\geq}& \estconst \left( \y{u_1 -u_2}^2_{H^1} + \y{z_1 -z_2}^2_{H^1}\right) +\kappa \y{\nabla z_1 - \nabla z_2}_{L^2}^2 \\
    \geq ~& \estconst \left( \y{u_1 -u_2}^2_{H^1} + \y{z_1 -z_2}^2_{H^1}\right).
\end{align*}
As a consequence, the claim
\begin{align*}
    \langle D \mathcal{E}(t,y_1(t)) &- D \mathcal{E}(t,y_2(t)), y_1(t)-y_2(t) \rangle  \\
    &\geq \min \left\lbrace \frac{2C}{p(2^{p-1}-1)}; \estconst \right\rbrace  \left(  [u_1 -u_2]_{W^{s,p}}^p +  \y{u_1 -u_2}^2_{H^1} + \y{z_1 -z_2}^2_{H^1} \right)  
\end{align*}
follows.
 \end{proof}

\section{Proofs of the main results}
\begin{proof}[Proof of \Cref{th:existence_of_solutions_continuum_system} (Existence of solutions to the continuum system)]
The idea is to apply Theorem \ref{th:existence of solutions}. Thus, it must be shown that $\Psi$ and $\mathcal{E}$ satisfy assumptions \eqref{assumptionI} -- \eqref{assumptionIV}.\\
\underline{\textbf{Assumption \eqref{assumptionI}:}} \\
By definition $\rho$ is positive, convex and positively homogeneous of degree $1$. As a consequence $\Psi (u,z) = \int \rho(z(x)) \dif x$ is 
$1$-homogeneous and satisfies conditions $(a)$ and $(b)$ in \eqref{assumptionI}. Lower semicontinuity of $\Psi$ follows by the convexity of $\rho$ and application of Fatou's lemma.\\
\\
\underline{\textbf{Assumption \eqref{assumptionII}} (\textit{Compactness of sublevels}):} \\
It needs to be shown that for any $\eta \in [0, \infty )$ the sublevel set 
\begin{align*}
    L_\eta^- \left( \mathcal{E}(t, \cdot ) \right) \coloneqq \left\lbrace y \in Y ~ \vert \mathcal{E}(t,y) \leq \eta \right\rbrace
\end{align*}
is compact. Therefore, take an arbitrary $\eta \in [0,\infty)$. 
Let $(y_i)_{i \in \mathbb{N}} = (u_i, z_i)_{i \in \mathbb{N}} $ be an arbitrary sequence in $ L_\eta^- \left( \mathcal{E}(t, \cdot ) \right) $. Then by \Cref{lemma:boundedness of norms} it follows that for every $i \in \mathbb{N}$
\begin{align} 
[u_i]_{W^{s,p}}, \y{ u_i }_{H^1},  \y{ z_i}_{H^1} < \infty.
\end{align}
By the Rellich--Kondrachov embedding theorem there exist subsequences $u_{i_k}$ and $z_{i_k}$ and limits $u \in L^p(Q)^d$ (and in particular in $L^2(Q) ^d$) and $ z \in L^2(Q)^{d\times d}$ such that $u_{i_k} \rightarrow u$ strongly in $L^p(Q)^d$ and $z_{i_k} \rightarrow z$ strongly in $L^2(Q)^{d\times d}$. It remains to show that the limit $y=(u,z)$ lies in the sublevelset $ L_\eta^- \left( \mathcal{E}(t, \cdot ) \right) $. Using that $\mathcal{E}$ is lower semi-continuous gives
\begin{align}\label{eq:closedness_of_sublevels}
    \mathcal{E} (t,y) \leq \liminf_{k \rightarrow \infty} \mathcal{E}(t, y_{i_k}) \leq \eta,
\end{align}
where the last inequality follows from the fact, that all members of the sequence $(y_{i_k})_k$ lie in the sublevelset $ L_\eta^- \left( \mathcal{E}(t, \cdot ) \right) $. Equation \eqref{eq:closedness_of_sublevels} says that $ y \in L_\eta^- \left( \mathcal{E}(t, \cdot ) \right)$ and hence it has been verified, that the sublevelsets of $\mathcal{E}(t,\cdot)$ are compact.\\
\\
\underline{\textbf{Assumption \eqref{assumptionIII}} (\textit{Energetic control of power}):} \\
Since $f \in C_{\text{Lip}}(I, L^2(Q)^d)$, the function $t \mapsto f(t,\cdot)$  is Lipschitz continuous from $I$ to $L^2(Q)^d$ and hence differentiable almost everywhere by Rademacher's theorem. This ensures that $\dot{f}(t,\cdot) := \partial_t f(t,\cdot)$ exists for a.e. $t \in I$ with $\y{\dot{f}(t, \cdot)}_{L^2(Q)^d} \leq C_{\text{Lip}}$ for some Lipschitz constant $C_{\text{Lip}} > 0$.
Consequently, 
$\partial_t \mathcal{E}(t,y) = - \int_Q \dot{f}(t,x) u(x) \dif x $ exists for all $ y \in \text{Dom} ~\mathcal{E}(0, \cdot ) $ 
and we can estimate 
\begin{align*}
    \Y{\partial_t \mathcal{E}(t,y) } &~= \Y{ \int_Q \dot{f}(t,x) u(x) \dif x } \leq \y{\dot{f}(t,\cdot)}_{L^2(Q)^d} \y{u}_{L^2(Q)^d} \leq C_{Lip} \y{u}_{L^2(Q)^d} \\
    &~= \under{\eqqcolon ~ c_\mathcal{E} ~< ~\infty}{(C_{Lip} + C_f) \y{u}_{L^2(Q)^d}} - C_f \y{u}_{L^2(Q)^d} \\
    &~\leq c_\mathcal{E} - C_f \y{u}_{L^2(Q)^d} +  \bar{c} \gagl[u] +\estconst \left( \y{ u }_{H^1(Q)^d}^2 + \y{ z}_{H^1(Q)^d}^2 \right) \overset{\eqref{eq:lower_estimate_energy}}{\leq}  c_\mathcal{E} + \mathcal{E}(t,y).
\end{align*}
\underline{\textbf{Assumption \eqref{assumptionIV}} (\textit{Conditions on convergent stable sequences}):} \\
Let $(t_n , y_n)_{n \in \mathbb{N}}$ be a stable sequence with $(t_n , y_n) \xlongrightarrow{I \times Y} (t,y) $. Then $u_n \rightarrow u$ in $L^2(Q)^d$ 
and by assumption $f \in C_{\text{Lip}}(I, L^2(Q)^d)$, which implies $\dot{f}(t, \cdot) \in L^2(Q) ^d$ for a.e. $t \in I$. This yields

\begin{align*}
    \partial_t \mathcal{E}(t,y_n) = - \int_Q \dot{f}(t,x) u_n(x) \dif x \xlongrightarrow{n\rightarrow \infty} - \int_Q \dot{f}(t,x) u(x) \dif x = \partial_t \mathcal{E}(t,y).
\end{align*}
It remains to show that $y \in \mathcal{S}(t)$. Note that
$(t_n , y_n) = (t_n, u_n, z_n) \xlongrightarrow{I \times Y} (t, u, z) = (t,y) $ means that 
$u_n \rightarrow u$ in $H^{1}(Q)^d$ 
and that $z_n \rightarrow z $ in $H^1(Q)^{d \times d} $. Since $\mathcal{E}^{loc}(u_{n}, z_{n}) $ consists of sums of strongly converging products this immediately implies that $\mathcal{E}^{loc}(y_n)  \xlongrightarrow{n\rightarrow \infty} \mathcal{E}^{loc}(y)$. 
Due to the Rellich--Kondrachov embedding theorem the convergence of $u_n$ to $u$ is valid in the strong $L^p(Q)^d$ topology up to subsequences and consequently
$\mathcal{E}^{V}(y_n) 
\xlongrightarrow{n\rightarrow \infty} \mathcal{E}^{V}(y)$ up to subsequences by 
\Cref{m:l:konvJeps}.
 By assumption $f \in C^{Lip}(I, L^2(Q))^d$ and $t_n \xlongrightarrow{n\rightarrow \infty} t $ from which it follows that $f(t_n) \xlongrightarrow{n\rightarrow \infty} f(t)$ in $L^2(Q)^d$. Together with the strong $L^2(Q)^d$ convergence of $u_n$ to $u$ it can be deduced that $\mathcal{F}(t_n, u_n) \xlongrightarrow{n\rightarrow \infty} \mathcal{F}(t,u)$. Due to the theorem of Bernstein-Doetsch (see for example \cite{Kuczma2009} Theorem 6.4.2.) $\rho $ is continuous. Moreover, strong $H^{1}(Q)^{d \times d}$ convergence of $z_n$ implies pointwise convergence of $z_n$ almost everywhere up to subsequences.
 For any testfunction $\tilde{y}=(\tilde{u},\tilde{z}) \in Y$ we therefore have $\rho (\tilde{z}-z_n) \xlongrightarrow{n\rightarrow \infty} \rho (\tilde{z}-z_n) $ pointwise almost everywhere. 
By assumption, $\rho$ is convex and therefore it is Lipschitz on $Q$ (see \Cref{lemma:convex_function_Lipschitz}). As a consequence $\rho(\tilde{z}(x) - z_n (x)) \leq L \Y{\tilde{z}(x) - z_n (x)} \in L^1(Q)$ and by the dominated convergence theorem it can be concluded that $\Psi (\tilde{y}- y_n)  \xlongrightarrow{n\rightarrow \infty} \Psi(\tilde{y} - y)$.

All the aforementioned convergences can be combined to
\begin{align*}
    0 \leq \lim_{n \rightarrow \infty} \mathcal{E}(t_n, \tilde{y} ) + \Psi (\tilde{y}- y_n) - \mathcal{E}(t_n,y_n) = \mathcal{E}(t,\tilde{y}) + \Psi(\tilde{y} - y) - \mathcal{E}(t,y).
\end{align*}
Rearranging this inequality results in $\mathcal{E}(t,y) \leq \mathcal{E}(t,\tilde{y}) + \Psi(\tilde{y} - y)  $ and it can be concluded that $ y \in \mathcal{S}(t)$.
\end{proof}

The subsequent proof for the uniqueness of solutions is based on the methods of Proposition 4.1 in \cite{MIELKE2} and Section 3.4.4 in \cite{mielkeRoubicek}.
\begin{proof}[Proof of \Cref{lemma:uniqueness_of_solutions} (Uniqueness of solutions)]
Let $y_1$ and $y_2$ be energetic solutions with $y_1(0) = y_2(0)=y_0$.
We define the function
\begin{align*}
    \gamma (t) := \langle D \mathcal{E}(t,y_1(t)) -D \mathcal{E}(t,y_2(t)), y_1(t)-y_2(t) \rangle.
\end{align*}
Note that the linear term $\mathcal{F}$ is canceled out here and $\gamma$ can be written as 
\begin{align*}
    \gamma (t)=\, &\langle D \mathcal{E}^V(t,y_1(t)) -D \mathcal{E}^V(t,y_2(t)), y_1(t)-y_2(t) \rangle \\
    &+ \langle D \mathcal{E}^{loc}(t,y_1(t)) -D \mathcal{E}^{loc}(t,y_2(t)), y_1(t)-y_2(t) \rangle.
\end{align*}
With this in mind, the time derivative of $\gamma$ can be estimated as
\begin{align*}
    \dot{\gamma}(t) =~& \langle D \mathcal{E}(y_1) -D\mathcal{E}(y_2), \dot{y}_1 -\dot{y}_2 \rangle + \langle D^2 \mathcal{E}(y_1) \dot{y}_1 -D^2 \mathcal{E}(y_2) \dot{y}_2, y_1 -y_2 \rangle \\
    \overset{\eqref{eq:mon1}}{\leq}&  \langle D^2 \mathcal{E}(y_1) \dot{y}_1 -D^2 \mathcal{E}(y_2) \dot{y}_2, y_1 -y_2 \rangle - \langle D \mathcal{E}(y_1) -D\mathcal{E}(y_2), \dot{y}_1 -\dot{y}_2 \rangle \\
    =~& \langle D\mathcal{E}(t,y_2) - D\mathcal{E}(t,y_1) - D^2\mathcal{E}(t,y_1)(y_2-y_1), \dot{y}_1  \rangle \\
    &+ \langle D\mathcal{E}(t,y_1) - D\mathcal{E}(t,y_2) - D^2\mathcal{E}(t,y_2)(y_1-y_2), \dot{y}_2 \rangle .
\end{align*}
Here it was used that $\langle D^2 \mathcal{E}(y_i) \dot{y}_i, y_j \rangle = \langle D^2 \mathcal{E}(y_i)  y_j, \dot{y}_i\rangle$ for $i,j \in \lbrace1,2\rbrace$. 
With $w \coloneqq u_1 - u_2$ and $v_\theta \coloneqq u_2 + \theta w$ we get by Taylor's theorem with integral remainder (or by application of the fundamental theorem of calculus twice)
\begin{align}\label{eq:lipsch_cond_sol}
     D\mathcal{E}(t,y_1) - D\mathcal{E}(t,y_2) - D^2\mathcal{E}(t,y_2)(y_1-y_2) = \int_0^1 (1-\theta)  D^3 \mathcal{E}(t,v_\theta)[w, w]\dif \theta
\end{align}
and thus
\begin{align}\label{eq:xi_ydot_est}
&    \langle D\mathcal{E}(t,y_1) - D\mathcal{E}(t,y_2) - D^2\mathcal{E}(t,y_2)(y_1-y_2), \dot{y}_2 \rangle \nonumber \\
    &\leq \int_0^1 (1-\theta)  \Y{\langle  D^3 \mathcal{E}(t,u_2+ \theta w)[w, w], \dot{u}_2\rangle} \dif \theta \nonumber\\
    &\leq C(p) \int_0^1 (1-\theta) \int_Q \int_Q \resizebox{0.6\textwidth}{!}{$\displaystyle K(x,y) \Y{v_\theta(x)-v_\theta(y)}^{p-3} \Y{\dot{u}_2(x) - \dot{u}_2(y)} \Y{w(x)-w(y)}^2 $} \dif x \dif y \dif \theta  \nonumber\\
    &\leq C(p) [\dot{u}_2]_{W^{s,p}}  [w]_{W^{s,p}}^2  \int _0^1 (1-\theta) [v_\theta]_{W^{s,p}}^{p-3}  \dif \theta,
\end{align}
where in the last step Hölders inequality was applied with $q_1=\frac{p}{p-3}$,  $q_2=p$, $q_3=\frac{p}{2}$.
Inserting $v_\theta = u_2 + \theta (u_1 - u_2)$ we can further estimate
\begin{align}\label{eq:int_theta}
    \int _0^1 (1-\theta) [v_\theta]_{W^{s,p}}^{p-3}  \dif \theta &\leq \int _0^1  \left([u_2]_{W^{s,p}} + \theta [u_1 - u_2]_{W^{s,p}}\right)^{p-3}  \dif \theta \nonumber\\
    &\leq \int _0^1  \left(2[u_2]_{W^{s,p}} +  [u_1]_{W^{s,p}}\right)^{p-3}  \dif \theta  \leq C .
\end{align}
As a consequence of \eqref{eq:lipschitz_estimate} we have 
\begin{align}\label{eq:der_dotu}
    [\dot{u}_i]_{W^{s,p}} \leq C. 
\end{align}
Combinig \eqref{eq:der_dotu} and \eqref{eq:int_theta} with \eqref{eq:xi_ydot_est} it can be deduced that
\begin{align*}
    \langle D\mathcal{E}(t,y_1) - D\mathcal{E}(t,y_2) - D^2\mathcal{E}(t,y_2)(y_1-y_2), \dot{y}_2 \rangle 
    \leq C(p) [u_1-u_2]_{W^{s,p}}^2 .
\end{align*}
Repeating the previous steps we get 
\begin{align*}
\langle D\mathcal{E}(t,y_2) - D\mathcal{E}(t,y_1) - D^2\mathcal{E}(t,y_1)(y_2-y_1), \dot{y}_1  \rangle \leq C(p) [u_1-u_2]_{W^{s,p}}^2
\end{align*}
and hence
\begin{align}
    \dot{\gamma}(t) \leq C  [u_1-u_2]_{W^{s,p}}^2 
    \overset{\eqref{eq:unif_convexity_est}}{\leq} C \gamma(t)^\frac{2}{p}.
\end{align}
Note that $\gamma(0) =0$ and therefore integration of the previous inequality results in
\begin{align*}
    \gamma(t) \leq C \int_0^t \gamma(s)^\frac{2}{p} \dif s.
\end{align*}
Then \Cref{lemma:LaSalle} is applicable and we get $\gamma(t) \leq 0$ and thus
\begin{align*}
    0 \leq C  \left(  [u_1 -u_2]_{W^{s,p}}^p +  \y{u_1 -u_2}^2_{H^1} + \y{z_1 -z_2}^2_{H^1} \right)  
    \overset{\eqref{eq:unif_convexity_est}}{\leq} \gamma(t) \leq 0,
\end{align*}
from which we are finally able to deduce that $\y{y_1 - y_2}_Y =0$ and therefore $y_1=y_2$.
\end{proof}

Next we present the proof of \Cref{thm:rate}, which is essentially the proof of Theorem 4.3 in \cite{neukamm2018stochastic} and of Theorem 4.3 in \cite{Mielke_twoScaleHomogenization} with minor changes due to the additional nonlocal term in the energy. 
\begin{proof}[Proof of Theorem \ref{thm:rate} (Convergence of solutions)]
Note that the existence of solutions $y_\varepsilon = (u_\varepsilon, z_\varepsilon)$ of \eqref{stability} and \eqref{energyBalance} associated with $\mathcal{E}_\varepsilon$ and $\Psi_\varepsilon$ follows by Theorem \ref{th:existence of discr solutions}, whereas existence and uniqueness of solutions to the continuum problem follow by \Cref{th:existence_of_solutions_continuum_system} and \Cref{lemma:uniqueness_of_solutions}.

We assume that $\mathcal{E} (t, y) < \infty$ for a solution $y=(u,z)$ of \eqref{stability} and \eqref{energyBalance} associated to $\mathcal{E}$ and $\Psi$ with $u (0)= u_0$ and $z(0) = z_0 $. We further assume that 
\begin{align}\label{as:energy_finite}
\sup_{\varepsilon} \mathcal{E}_\varepsilon (t, y_\varepsilon) < \infty  \qquad  \forall t \in I,
\end{align}
where $y_\varepsilon = (u_\varepsilon ,z_\varepsilon)$ are solutions to \eqref{stability} and \eqref{energyBalance} associated to $\mathcal{E}_\varepsilon$ and $\Psi_\varepsilon$ with $u_\varepsilon (0)= u_\varepsilon^0$ and $z_\varepsilon (0) = z_\varepsilon^0 $
(otherwise the statement is trivial).\\
Rewriting the discrete energy $\mathcal{E}_\varepsilon (y_\varepsilon)$ in integral form as in \eqref{eq:integral_representation_discrete_energy} and applying \Cref{lemma:boundedness of norms} we get 
\begin{align*}
    \sup_{\varepsilon} \, \y{\Repsu}_{L^2(Q)^d} < \infty,&\qquad \sup_{\varepsilon} \,\y{\Reps \nabla_\varepsilon u_\varepsilon}_{L^2(Q)^{d\times d } } < \infty, \\
    &~\text{and} \\
     \sup_{\varepsilon} \,\y{\Reps z_\varepsilon}_{L^2(Q)^{d \times d}} < \infty ,&  \qquad \sup_{\varepsilon} \,\y{\Reps \nabla_\varepsilon z_\varepsilon}_{L^2(Q)^{d \times d \times d}} < \infty.
\end{align*}
As a consequence, there exist subsequences (not relabeled) and weak limits $u \in L^2(Q) ^d $, $z \in L^2(Q)^{d \times d}$ such that
\begin{align}\label{eq:conv_properties1}
    \mathcal{R}^\ast_{\varepsilon} u_{\varepsilon} \xrightharpoonup{L^2(Q)^{d}} u ,& \qquad \mathcal{R}^\ast_{\varepsilon} \nabla_{\varepsilon} u_{\varepsilon} \xrightharpoonup{L^2(Q)^{d\times d}} \nabla u \qquad ~~\text{weakly  }\\
    &~\text{and} \nonumber \\
    \mathcal{R}^\ast_{\varepsilon} z_{\varepsilon} \xrightharpoonup{L^2(Q)^{d\times d}} z ,& \qquad \mathcal{R}^\ast_{\varepsilon} \nabla_{\varepsilon} z_{\varepsilon} \xrightharpoonup{L^2(Q)^{d\times d \times d}} \nabla z \qquad\text{weakly  }. \label{eq:conv_properties2}
\end{align}
Now we have to show that $(u(t),z(t)) \in \mathcal{S}(t)$ for any $t \in I$. \\
For any $\tilde{y}=(\tilde{u}, \tilde{z}) \in H^1_0(Q)^d \times H^1(Q)^{d \times d}$, we define the sequences $$v_\varepsilon \coloneqq \mathcal{R}_\varepsilon (\tilde{u} - u), ~w_\varepsilon \coloneqq \mathcal{R}_\varepsilon (\tilde{z}-z), ~
\tilde{u}_\varepsilon \coloneqq u_\varepsilon + v_\varepsilon ~ \text{and} ~\tilde{z}_\varepsilon \coloneqq z_\varepsilon + w_\varepsilon ,$$ which obviously converge in the following sense:
\begin{align} \label{eq:conv_properties}
\begin{array}{lll}
\Reps v_\varepsilon \rightarrow \tilde{u}-u ~\text{in } L^2(Q)^d,&  \Reps \nabla_\varepsilon v_\varepsilon \rightarrow \nabla( \tilde{u}-u) ~\text{in } L^2(Q)^{d \times d}  &\text{strongly}, \\
& & \\
\Reps w_\varepsilon \rightarrow \tilde{z} - z ~\text{in } L^2(Q)^{d \times d} ,\qquad&  \Reps \nabla_\varepsilon  w_\varepsilon \rightarrow \nabla ( \tilde{z} - z )~\text{in } L^2(Q)^{d \times d \times d} \qquad & \text{strongly},\\
& & \\
\Reps \tilde{u}_\varepsilon \rightharpoonup \tilde{u} ~\text{in } L^2(Q)^d, &  \Reps \nabla_\varepsilon \tilde{u}_\varepsilon \rightharpoonup \nabla \tilde{u} ~\text{in } L^2(Q)^{d \times d} & \text{weakly},\\
& & \\
\Reps \tilde{z}_\varepsilon \rightharpoonup \tilde{z} ~\text{in } L^2(Q)^{d \times d},&  \Reps \nabla_\varepsilon \tilde{z}_\varepsilon \rightharpoonup \nabla  \tilde{z} ~\text{in } L^2(Q)^{d \times d \times d} & \text{weakly}.
\end{array}
\end{align}
We want to show, that these sequences satisfy
\begin{align*}
\lim_{\varepsilon \rightarrow 0}\left( \mathcal{E}_\varepsilon (t, \tilde{y}_\varepsilon ) + \Psi_\varepsilon (\tilde{y}_\varepsilon - y_\varepsilon) - \mathcal{E}_\varepsilon (t, y_\varepsilon) \right) = \mathcal{E} (t,\tilde{y}) + \Psi (\tilde{y}-y) - \mathcal{E}(t,y) \qquad \text{a.s.}
\end{align*}
This would yield $\mathcal{E}(t,y) \leq \mathcal{E} (t,\tilde{y}) + \Psi (\tilde{y}-y)$, and therefore $y(t)=(u(t),z(t)) \in \mathcal{S}(t)$.
First we consider
\begin{align*}
\mathcal{E}_\varepsilon (t, \tilde{y}_\varepsilon) - \mathcal{E}_\varepsilon (t,y_\varepsilon) =& \mathcal{E}_\varepsilon^{loc}(\tilde{y}_\varepsilon) - \mathcal{E}^{loc}_\varepsilon (y_\varepsilon) + \mathcal{E}_\varepsilon^{V}( \tilde{u}_\varepsilon) - \mathcal{E}^{V}_\varepsilon (u_\varepsilon) \\
 &- \varepsilon^d \sum_{x\in Q_\varepsilon} f_\varepsilon(t,x)(\tilde{u}_\varepsilon - u_\varepsilon)(x)  .
\end{align*}
Using the identities $(a:\bold{A}:a) -(b:\bold{A}:b) = (a-b) :\bold{A} :(a+b)$ and $\Y{a}^2 - \Y{b}^2 = (a-b) : (a+b)$,
the first two terms can be written as
\begin{align*}
&\mathcal{E}_\varepsilon^{loc}(\tilde{y}_\varepsilon) - \mathcal{E}^{loc}_\varepsilon (y_\varepsilon) \\
& = 
 \varepsilon^{d} \sum_{x\in \Z_\varepsilon^d} \left( \left(\nabla ^s_\varepsilon \tilde{u}_\varepsilon-  \nabla ^s_\varepsilon u_\varepsilon\right)
 - \left(\tilde{z}_\varepsilon- z_\varepsilon\right) \right): \bold{A}:  \left( \left(\nabla ^s_\varepsilon \tilde{u}_\varepsilon+  \nabla ^s_\varepsilon u_\varepsilon\right)
 - \left(\tilde{z}_\varepsilon + z_\varepsilon\right) \right)  \\
  &\qquad+ \left(\tilde{z}_\varepsilon - z_\varepsilon\right) : \bold{H} : \left(\tilde{z}_\varepsilon + z_\varepsilon\right)   
  + \kappa \left( \nabla_\varepsilon \tilde{z}_\varepsilon  - \nabla_\varepsilon z_\varepsilon \right): \left( \nabla_\varepsilon \tilde{z}_\varepsilon  + \nabla_\varepsilon z_\varepsilon \right)\\
  &=  \varepsilon^{d} \sum_{x\in \Z_\varepsilon^d} \left( \nabla ^s_\varepsilon v_\varepsilon 
 -  w_\varepsilon \right): \bold{A}:  \left( \left(\nabla ^s_\varepsilon \tilde{u}_\varepsilon+  \nabla ^s_\varepsilon u_\varepsilon\right)
 - \left(\tilde{z}_\varepsilon(x) + z_\varepsilon \right) \right) \\
  &\qquad+ w_\varepsilon   : \bold{H} : \left(\tilde{z}_\varepsilon + z_\varepsilon \right) + \kappa ~ \nabla_\varepsilon w_\varepsilon : \left( \nabla_\varepsilon \tilde{z}_\varepsilon + \nabla_\varepsilon z_\varepsilon \right) .
\end{align*}
Here the dependence of $u_\varepsilon, \tilde{u}_\varepsilon, z_\varepsilon, \tilde{z}_\varepsilon, v_\varepsilon, w_\varepsilon, \bold{A, \bold{H}}$ on $x$ has been omitted for better readability.
By \eqref{eq:conv_properties1} and \eqref{eq:conv_properties}, the above expression involves sums of products of strongly and weakly convergent sequences and hence its limit can be determined as
\begin{align*}
\lim&_{\varepsilon \rightarrow 0 } \mathcal{E}_\varepsilon^{loc}(\tilde{y}_\varepsilon) - \mathcal{E}^{loc}_\varepsilon (y_\varepsilon) \\
=& \int_Q \left(\left(\nabla ^s  \tilde{u}- \nabla ^s u\right) -\left( \tilde{z}-z\right) \right) : \bold{A}(x) : \left(\left(\nabla ^s \tilde{u} + \nabla ^s u \right) - \left(\tilde{z}+ z \right) \right) \\
&+ \left(\tilde{z}-z \right) : \bold{H}(x) : \left( \tilde{z}+ z\right) + \kappa \left(\nabla \tilde{z}- \nabla z \right) : \left( \nabla \tilde{z} + \nabla z \right) \dif x \\
=& ~\mathcal{E}^{loc}(\tilde{y}) - \mathcal{E}^{loc} (y).
\end{align*}
To show convergence of $\mathcal{E}_\varepsilon^{V}( \tilde{u}_\varepsilon) - \mathcal{E}^{V}_\varepsilon (u_\varepsilon) $, we argue as follows. Due to the assumption made in \eqref{as:energy_finite}, we get that 
\begin{align*}
\sup_{\varepsilon} \mathcal{E}^{loc}_\varepsilon (y_\varepsilon) < \infty .
\end{align*}
Combining this with \Cref{th:Korn_plasticity} (Korn's inequality), it follows that
\begin{align*}
\varepsilon^d \sum_{x \in Q_\varepsilon} 
\Y{\nabla_\varepsilon u_\varepsilon (x) }^2
< \infty .
\end{align*}
This allows application of Theorem \ref{thm:Poincare}, which says that $\Repsu \rightarrow u$ strongly in $L^p(Q)^d$ (and in particular in $L^2(Q)^d$) up to subsequences. Then Theorem \ref{m:l:konvJeps} yields 
\begin{align*}
    \mathcal{E}_{\varepsilon}^V(\tilde{u}_\varepsilon)-\mathcal{E}_\varepsilon^V(u_\varepsilon ) \asarrow \mathcal{E}^V(\tilde{u})-\mathcal{E}^V(u).
\end{align*}
The last term obviously converges to $- \int_Q f(t,x) (\tilde{u}-u)(x) \dif x$. 
Putting the previous arguments together we get 
$$\lim_{\varepsilon \rightarrow 0}\left( \mathcal{E}_\varepsilon(t,\tilde{y}_\varepsilon) - \mathcal{E}_\varepsilon (t,y_\varepsilon) \right) = \mathcal{E}(t,\tilde{y}) - \mathcal{E}(t,y)\qquad \text{a.s.}$$
It remains to show that $\lim_{\varepsilon \rightarrow 0} \Psi_\varepsilon (\tilde{y}_\varepsilon - y_\varepsilon) = \Psi (\tilde{y}-y)$.
We have
\begin{align*}
\Psi_\varepsilon (\tilde{y}_\varepsilon-y_\varepsilon) &= \Psi_\varepsilon (w_\varepsilon ) = \varepsilon^d \sum_{x\in \Z_\varepsilon^d} \rho \left(\mathcal{R}_\varepsilon (\tilde{z}-z) (x)\right) \overset{\text{Jensen}}{\leq} \varepsilon^d \sum_{x\in \Z_\varepsilon^d} \mathcal{R}_\varepsilon \rho (\tilde{z}-z) (x)\\
 &= \int_{\R^d}  \Reps  \mathcal{R}_\varepsilon \rho (\tilde{z}-z) (x)  \dif x \xlongrightarrow{\varepsilon \rightarrow 0} \int_Q \rho (\tilde{z}-z) (x)  \dif x =\Psi (\tilde{z}-z).
\end{align*}
The other inequality can be obtained by 
\begin{align*}
\liminf_{\varepsilon \rightarrow 0}& \Psi_\varepsilon (\tilde{y}_\varepsilon-y_\varepsilon) = \liminf_{\varepsilon \rightarrow 0} \Psi_\varepsilon (w_\varepsilon ) = \liminf_{\varepsilon \rightarrow 0} \varepsilon^d \sum_{x\in \Z_\varepsilon^d} \rho \left(\mathcal{R}_\varepsilon (\tilde{z}-z) (x)\right)  \\
& =\liminf_{\varepsilon \rightarrow 0} \int_{\R^d} \rho \left(\Reps \mathcal{R}_\varepsilon (\tilde{z}-z) (x)\right)
\overset{\text{Fatou}}{\geq} \int_{\R^d} \liminf_{\varepsilon \rightarrow 0} \rho \left(\Reps \mathcal{R}_\varepsilon (\tilde{z}-z) (x)\right) \\ 
&= \int_Q \rho (\tilde{z}-z) (x)  \dif x =\Psi (\tilde{z}-z).
\end{align*}
With this we finally obtain $y(t)=(u(t),z(t)) \in \mathcal{S}(t)$.
\\

In the next step we want to show that in the limit $\varepsilon \rightarrow 0$ the energy balance equation \eqref{energyBalance} holds, that reads
\begin{align}\label{eq:energy balance}
\mathcal{E}(t,y(t)) + \int_0^t \Psi(\dot{y}(s) ) \dif s = \mathcal{E}(0,y(0))  - \int_0^t \int_Q \dot{f}(s)\cdot u(s) \dif x \dif s,
\end{align}
where $y(t)=(u(t),z(t))$.
The energy balance equation \eqref{energyBalance} associated to $\mathcal{E}_\varepsilon$ and $\Psi_\varepsilon$ is given by 
\begin{align}\label{eq:rate:E}
&\mathcal{E}_\varepsilon(t,y_\varepsilon(t)) + \int_0^t \Psi_\varepsilon(\dot{y}_\varepsilon(s) ) \dif s \\ \nonumber
& =  \int_Q \int_Q \Reps \coeff \frac{\Y{\left(\Repsu^0 (x) - \Repsu^0 (y)\right)\cdot\tfrac{\Reps x-\Reps y}{\Y{\Reps x- \Reps y}}}^p}{\Y{\Reps x- \Reps y}^{d+ps}}  \dif x \dif y \\ \nonumber
&\qquad+ \int_Q \left( \nabla^s_\varepsilon \Repsu^0 (x) - \Reps z_\varepsilon^0 (x) \right) : \bold{A}(x) : \left( \nabla^s_\varepsilon \Repsu^0 (x) - \Reps z_\varepsilon^0 (x) \right) \\ \nonumber
&\qquad ~~~~~~+ \Reps z_\varepsilon^0 (x) : \bold{H}(x) : \Reps z_\varepsilon^0 (x) + \kappa \Y{\nabla_\varepsilon \Reps z_\varepsilon^0 (x)}^2 \dif x  \\
&\qquad- \int_Q \Reps f_\varepsilon (0,x) \Repsu^0  \dif x   - \int_0^t  \varepsilon^d \sum_{x\in \Z^d_\varepsilon} \dot{f}_\varepsilon(s,x) u_\varepsilon(s,x) \dif s . \nonumber
\end{align}
By the strong convergences in \eqref{eq:thm:rate:conv1} we get, that the right-hand side of the above equation except of the last term converges to 
\begin{align}\label{eq:lim_e0}
\mathcal{E}(0,y(0))&= \bar{c} 
\int_Q \int_Q \frac{\Y{u^0(x)-u^0(y)\cdot \tfrac{x-y}{\Y{x-y}}}^p}{\Y{x-y}^{d+ps}} \dif x \dif y \\ \nonumber
&\qquad + \int_Q \left(\nabla ^s u^0(x) - z^0(x) \right) : \bold{A}(x) : \left(\nabla ^s u^0(x) - z^0(x) \right) \\
&\qquad+ z^0(x) : \bold{H}(x) : z^0(x) + \kappa \Y{\nabla z^0(x)}^2 \dif x  \nonumber \\ \nonumber
& \qquad +\int_{Q} f(0,x) u^0(x) \dif x .
\end{align}
The last term on the right-hand side of \eqref{eq:rate:E} can be written as
\begin{align}
\label{eq:term_right}
\int_0^t  \int_Q \Reps\dot{f}_\varepsilon(s,x) ~\Reps u_\varepsilon(s,x)  \dif x \dif s
\end{align}
and due to the uniform boundedness of $\Reps\dot{f}_\varepsilon(s,x) ~\Reps u_\varepsilon(s,x) $ on $Q$, the dominated convergence theorem is applicable and yields convergence of \eqref{eq:term_right} to 
\begin{align}\label{eq:lim_f}
\int_0^t \int_Q \dot{f}(s,x) u(s,x) \dif x \dif s.
\end{align} 
 Now we consider the left-hand side of the equation \eqref{eq:rate:E}.
Being a convex, strongly continuous functional, $\mathcal{E}^{loc}$ is weakly lower semicontinuous and thus we get the $\liminf$-inequality for the local term of the energy. 
As mentioned before, the convergence $\Repsu \rightarrow u$ holds strongly in $L^p(Q)^d$ up to subsequences.
Thus, we get the $\liminf$-inequality for the non-local term (up to subsequences) by \Cref{m:l:konvJeps}, which leads to the $\liminf$-inequality for the whole energy:
\begin{align} \label{eq:liminf_nonloc}
 \mathcal{E}(t,y(t)) \leq \liminf_{\varepsilon\rightarrow 0} \mathcal{E}_{\varepsilon}(t,y_{\varepsilon}(t)) .
\end{align}
To prove the $\liminf$-inequality for $\int_0^t \Psi_\varepsilon(\dot{y}_\varepsilon(t) ) \dif t$ we choose an arbitrary partition $\lbrace t_i\rbrace$ of the interval $I$, where $0=t_1 < t_2< ... < t_N=T$ . For any sequence $\varepsilon \rightarrow 0$ (in particular for any (sub-)sub-sequences) it holds
$$ \sum_{i=2}^N \Psi(y(t_i)-y(t_{i-1})) \leq \liminf_{\varepsilon \rightarrow 0} \sum_{i=2}^N \Psi_\varepsilon(y_\varepsilon(t_i)-y_\varepsilon(t_{i-1})) .$$
From this inequality we get by the homogeneity of $\Psi$ and $\Psi_\varepsilon$:
$$ \sum_{i=2}^N \Y{t_i-t_{i-1}} \Psi\left(\frac{y(t_i)-y(t_{i-1})}{\Y{t_i-t_{i-1}} }\right) \leq \liminf_{\varepsilon \rightarrow 0} \sum_{i=2}^N \Y{t_i-t_{i-1}} \Psi_\varepsilon\left(\frac{y_\varepsilon(t_i)-y_\varepsilon(t_{i-1})}{\Y{t_i-t_{i-1}}}\right). $$
Because this inequality holds for any partition $\lbrace t_i\rbrace$ of $I$, taking the supremum over all partitions of $I$, this turns into
\begin{align}\label{eq:liminf Psi}
\int_0^t \Psi (\dot{y}(s)) \dif s \leq \liminf_{\varepsilon \rightarrow 0} \int_0^t \Psi_\varepsilon(\dot{y}_\varepsilon(s))\dif s.
\end{align}
Combining the previous results with \eqref{eq:rate:E}, \eqref{eq:lim_e0} and \eqref{eq:lim_f} we get:
\begin{align*}
\mathcal{E}(t,y(t)) + \int_0^t \Psi(\dot{y}(s) ) \dif s &\leq
\liminf_{\varepsilon \rightarrow 0} \mathcal{E}_\varepsilon (t,y(t)) + \Psi_\varepsilon(\dot{y}_\varepsilon(s))\dif s \\
&=\mathcal{E}(0,y(0)) - \int_0^t \int_Q \dot{f}(s)  u(s) \dif x \dif s.
\end{align*}
To prove the other inequality, we proceed as in Section 2.3.1 of \cite{MIELKE2} and we use that $y$ satisfies the stability condition
\begin{align*}\mathcal{E}(t,y) \leq \mathcal{E} (t,\tilde{y}) + \Psi (\tilde{y}-y) \qquad \forall \tilde{y}=(\tilde{u}, \tilde{z}) \in Y 
.
\end{align*}
We take a sequence of partitions $\lbrace t_i^N\rbrace$ of the interval $I$, where $0=t_1^N < t_2^N< ... < t_N^N=t$ such that the fineness $\sup_{i=1,...,N}(t_j^N -t_{i-1}^N)$ tends to zero as $N \rightarrow \infty$. Then we get by the previous stability condition for any  $i\geq 1$
\begin{align*}
\mathcal{E}(&t_i, y(t_i))+\Psi \left(y(t_i)-y(t_{i-1})\right) \\
 &= \int_{t_{i-1}}^{t_i} \partial_s \mathcal{E}(s, y(t_i)) \dif s + \mathcal{E}(t_{i-1},y(t_i))+ \Psi \left(y(t_i)-y(t_{i-1})\right) \\
&\geq \int_{t_{i-1}}^{t_i} \partial_s \mathcal{E}(s, y(t_i)) \dif s + \mathcal{E}(t_{i-1},y(t_{i-1})) \\
 &=- \int_{t_{i-1}}^{t_i} \int_Q \dot{f}(s)\cdot u(s) \dif x \dif s + \mathcal{E}(t_{i-1},y(t_{i-1})) .
\end{align*}
Summation over $i=1,...,N$ gives
\begin{align*}
\mathcal{E}(t,y(t)) +& \int_0^t \Psi(\dot{y}(s) ) \dif s  \geq \mathcal{E}(t,y(t))  + \sum_{i=1}^N \Psi \left(y(t_i)-y(t_{i-1})\right) \\
 &\geq - \sum_{i=1}^N  \int_{t_{i-1}}^{t_i} \int_Q \dot{f}(s)\cdot u(s) \dif x \dif s  +  \mathcal{E}(0,y(0)) \\
&\xlongrightarrow{N \rightarrow \infty} \mathcal{E}(0,y(0)) - \int_0^t \int_Q \dot{f}(s)\cdot u(s) \dif x \dif s,
\end{align*}
and with this we finally get the energy balance equation \eqref{eq:energy balance}.
This allows us to pass to the limit:
\begin{align}\label{eq:poof rate}
\lim_{\varepsilon \rightarrow 0} &\left( \mathcal{E}_\varepsilon (t, y_\varepsilon (t)) + \int_0^t \Psi_\varepsilon (\dot{y}_\varepsilon (s)) \dif s  \right) \\
\nonumber
&= \lim_{\varepsilon \rightarrow 0} \left( \mathcal{E}_\varepsilon(0,y_\varepsilon
(0)) - \int_0^t  \varepsilon^d \sum_{x\in \Z^d_\varepsilon} \dot{f}_\varepsilon(s,x) u_\varepsilon(s,x) \dif s\right) \\
\nonumber
&= \mathcal{E}(0,y(0)) - \int_0 ^t \int_{Q}\dot{f} (s,x) \cdot u (s,x) \dif x \dif s \\
\nonumber
&=\mathcal{E}(t,y(t)) +\int_0^t \Psi (\dot{y}(s)) \dif s .
\end{align}
In the last step we have to show that the convergence of $y_\varepsilon =( u_\varepsilon , z_\varepsilon )$ in \eqref{eq:conv_properties1} and \eqref{eq:conv_properties2} does not only hold in the weak but also in the strong sense.
Therefore we construct the recovery sequence
$\hat{y}_\varepsilon \coloneqq (\Raeps u, \Raeps z) $, that converges strongly in the following sense:
\begin{align*}
    \Reps \hat{y}_\varepsilon = (\Reps \Raeps u, \Reps \Raeps z)  &\rightarrow y =(u,z) &&\quad \text{in } L^2(Q)^d \times L^2(Q)^{d \times d} \\
    \Reps \nabla_\varepsilon \hat{y}_\varepsilon = (\Reps \nabla_\varepsilon \Raeps u, \Reps \nabla_\varepsilon \Raeps z ) &\rightarrow \nabla y = (\nabla u, \nabla z) &&\quad \text{in } L^2(Q)^{d\times d} \times L^2(Q)^{d \times d \times d}
\end{align*}
With this we get (omitting the dependence of $u_\varepsilon, \hat{u}_\varepsilon, z_\varepsilon, \hat{z}_\varepsilon, \bold{A, \bold{H}}$ on $x$  for better readability)
\begin{align}\nonumber
&\mathcal{E}^{loc}_\varepsilon (t,y_\varepsilon)- \mathcal{E}^{loc}_\varepsilon (t, \hat{y}_\varepsilon) 
- 2 \varepsilon^{d} \sum_{x\in Q_\varepsilon^d} \left[ \nabla ^s_\varepsilon \hat{u}_\varepsilon : \bold{A}:  \left( \nabla ^s_\varepsilon u_\varepsilon - \nabla ^s_\varepsilon \hat{u}_\varepsilon 
 - \left( z_\varepsilon - \hat{z}_\varepsilon \right) \right) \right]\\ \nonumber
  &\qquad - 2 \varepsilon^{d} \sum_{x\in Q_\varepsilon^d} \left[ \left(z_\varepsilon - \hat{z}_\varepsilon\right) : \bold{H} :  \hat{z}_\varepsilon + \kappa \nabla_\varepsilon \hat{z}_\varepsilon : \left(\nabla_\varepsilon z_\varepsilon - \nabla_\varepsilon \hat{z}_\varepsilon \right)\right] \\ \nonumber
 &=  \varepsilon^{d} \sum_{x\in Q_\varepsilon^d} \left[ \left( \nabla ^s_\varepsilon u_\varepsilon - \nabla ^s_\varepsilon \hat{u}_\varepsilon
 - \left( z_\varepsilon - \hat{z}_\varepsilon \right) \right): \bold{A}:  \left( \nabla ^s_\varepsilon u_\varepsilon - \nabla ^s_\varepsilon \hat{u}_\varepsilon 
 - \left( z_\varepsilon - \hat{z}_\varepsilon \right) \right) \right] \\ \nonumber
& \qquad + \varepsilon^{d} \sum_{x\in Q_\varepsilon^d} \left[ \left(z_\varepsilon - \hat{z}_\varepsilon \right) : \bold{H} :   \left(z_\varepsilon - \hat{z}_\varepsilon \right) + \kappa \left( \nabla_\varepsilon z_\varepsilon - \nabla_\varepsilon \hat{z}_\varepsilon \right):\left( \nabla_\varepsilon z_\varepsilon - \nabla_\varepsilon \hat{z}_\varepsilon \right) \right] \\ \nonumber
&\geq \lambda_{\bold{A}}^{min} \y{\nabla ^s_\varepsilon u_\varepsilon - \nabla ^s_\varepsilon \hat{u}_\varepsilon
 - \left( z_\varepsilon - \hat{z}_\varepsilon \right) }_{\ell^2}^2 + \lambda_{\bold{H}}^{min} \y{ z_\varepsilon - \hat{z}_\varepsilon }_{\ell^2}^2 + \kappa \y{\nabla_\varepsilon z_\varepsilon - \nabla_\varepsilon \hat{z}_\varepsilon }_{\ell^2} \\ \nonumber
& \geq  \lambda_{\bold{A}}^{min} \left( \tfrac{1}{2}\y{\nabla ^s_\varepsilon u_\varepsilon - \nabla ^s_\varepsilon \hat{u}_\varepsilon}_{\ell^2}^2 - \y{ z_\varepsilon - \hat{z}_\varepsilon }_{\ell^2}^2 \right) + \lambda_{\bold{H}}^{min} \y{ z_\varepsilon - \hat{z}_\varepsilon }_{\ell^2}^2  +\kappa \y{\nabla_\varepsilon z_\varepsilon - \nabla_\varepsilon \hat{z}_\varepsilon }_{\ell^2} \\ \nonumber
&= \tfrac{\lambda_{\bold{A}}^{min}}{2}\y{\nabla ^s_\varepsilon u_\varepsilon - \nabla ^s_\varepsilon \hat{u}_\varepsilon}_{\ell^2}^2 + \left( \lambda_{\bold{H}}^{min}- \lambda_{\bold{A}}^{min} \right) \y{ z_\varepsilon - \hat{z}_\varepsilon }_{\ell^2}^2 + \kappa \y{\nabla_\varepsilon z_\varepsilon - \nabla_\varepsilon \hat{z}_\varepsilon }_{\ell^2}\\ 
&\geq  \min \left(\tfrac{\lambda_{\bold{A}}^{min}}{2} , \lambda_{\bold{H}}^{min}- \lambda_{\bold{A}}^{min}, \kappa \right) \left( \y{\nabla ^s_\varepsilon u_\varepsilon - \nabla ^s_\varepsilon \hat{u}_\varepsilon}_{\ell^2}^2 + \y{ z_\varepsilon - \hat{z}_\varepsilon }_{\ell^2}^2 +   \y{\nabla_\varepsilon z_\varepsilon - \nabla_\varepsilon \hat{z}_\varepsilon }_{\ell^2}\right). \label{eq:norm:conv}
\end{align}
The last two terms on the left side of the above equations converge as products of weakly and strongly convergent sequences. Regarding convergence of the first two terms we proceed as follows. From  \eqref{eq:poof rate} we deduce that 
\begin{align*}
\limsup_{\varepsilon \rightarrow 0} \mathcal{E}_\varepsilon (t, y_\varepsilon (t)) + \liminf_{\varepsilon \rightarrow 0} \int_0^t \Psi_\varepsilon (\dot{y}_\varepsilon (s)) \dif s  
=\mathcal{E}(t,y(t)) +\int_0^t \Psi (\dot{y}(s)) \dif s
\end{align*}
and in combination with \eqref{eq:liminf Psi} we get
\begin{align*}
\limsup_{\varepsilon \rightarrow 0 } \mathcal{E}^{loc}_\varepsilon (t, y_\varepsilon) &~\leq \mathcal{E}^{loc} (t, y) + \mathcal{E}^{V} (t, y) - \mathcal{F}(t,y) +\int_0^t \Psi (\dot{y}(s)) \dif s \\
&\qquad - \limsup_{\varepsilon \rightarrow 0 } \mathcal{E}^{V}_\varepsilon (t, y_\varepsilon) + \limsup_{\varepsilon \rightarrow 0 } \mathcal{F}_\varepsilon (t, y_\varepsilon) - \liminf_{\varepsilon \rightarrow 0 } \int_0^t \Psi_\varepsilon (\dot{y}_\varepsilon (s)) \dif s \\
&\overset{\eqref{eq:liminf Psi}}{\leq} \mathcal{E}^{loc} (t, y) + \mathcal{E}^{V} (t, y) - \limsup_{\varepsilon \rightarrow 0 } \mathcal{E}^{V}_\varepsilon (t, y_\varepsilon) - \mathcal{F}(t,y)  + \limsup_{\varepsilon \rightarrow 0 } \mathcal{F}_\varepsilon (t, y_\varepsilon) .
\end{align*}
The last term on the right side of the inequality obviously converges to $\mathcal{F}(t,y)$. With the same argumentation as for equation \eqref{eq:liminf_nonloc} convergence of $\limsup_{\varepsilon \rightarrow 0 } \mathcal{E}^{V}_\varepsilon (t, y_\varepsilon)$ to $ \mathcal{E}^{V} (t, y)$ (up to subsequences) follows.
Hence the previous inequality reduces to
\begin{align*}
\limsup_{\varepsilon \rightarrow 0 } \mathcal{E}^{loc}_\varepsilon (t, y_\varepsilon) \leq  \mathcal{E}^{loc} (t, y).
\end{align*}
Therefore for the first two terms on the left side of \eqref{eq:norm:conv} it holds 
\begin{align*}
\lim_{\varepsilon \rightarrow 0 } \mathcal{E}^{loc}_\varepsilon (t,y_\varepsilon)- \mathcal{E}^{loc}_\varepsilon (t, \hat{y}_\varepsilon)  \leq 0
\end{align*}
and finally we can deduce from  \eqref{eq:norm:conv} that 
\begin{align*}
 \y{\nabla ^s_\varepsilon u_\varepsilon - \nabla ^s_\varepsilon \hat{u}_\varepsilon}_{\ell^2(Q_\varepsilon)^d}^2 &\epsarrow  0,  \\
  \y{ z_\varepsilon - \hat{z}_\varepsilon }_{\ell^2(Q_\varepsilon)^d}^2  &\epsarrow  0, \\
   \y{\nabla _\varepsilon z_\varepsilon - \nabla _\varepsilon \hat{z}_\varepsilon}_{\ell^2(Q_\varepsilon)^d}^2 &\epsarrow  0 .
\end{align*}
We get the desired strong convergence by 
\begin{align*}
&\y{ u - \Reps u_\varepsilon }_{L^2}^2 + \y{ z - \Reps z_\varepsilon }_{L^2}^2 + \y{\nabla u - \Reps \nabla _\varepsilon  u_\varepsilon}_{L^2}^2  + \y{\nabla z - \Reps \nabla _\varepsilon  z_\varepsilon}_{L^2}^2 \\
& \hspace{4pt}\leq C \left( \y{\nabla u - \Reps \nabla _\varepsilon  u_\varepsilon}_{L^2}^2  + \y{ z - \Reps z_\varepsilon}_{L^2}^2 + \y{\nabla z - \Reps \nabla _\varepsilon  z_\varepsilon}_{L^2}^2 \right) \\
&\overset{\text{Korn}}{\leq} C \left( \y{\nabla ^s u - \Reps \nabla ^s_\varepsilon  u_\varepsilon}_{L^2}^2 
+ \y{ z - \Reps z_\varepsilon}_{L^2}^2 + \y{\nabla z - \Reps \nabla _\varepsilon  z_\varepsilon}_{L^2}^2 \right) \\
 & \hspace{4pt}\leq C \Big(
  \y{\nabla ^s u - \Reps \nabla ^s_\varepsilon \hat{u}_\varepsilon}_{L^2}^2    
+  \y{\Reps \nabla ^s_\varepsilon \hat{u}_\varepsilon - \Reps \nabla ^s_\varepsilon u_\varepsilon}_{L^2}^2  + \y{z - \Reps  \hat{z}_\varepsilon}_{L^2}^2 
\\ & \qquad 
+  \y{\Reps \hat{z}_\varepsilon - \Reps z_\varepsilon}_{L^2}^2 
 + \y{\nabla z - \Reps \nabla_\varepsilon \hat{z}_\varepsilon}_{L^2}^2 
+  \y{\Reps \nabla_\varepsilon \hat{z}_\varepsilon - \Reps \nabla_\varepsilon z_\varepsilon}_{L^2}^2 \Big)  \epsarrow 0.
\end{align*}
Recall that we have extracted a converging subsequence $y_{\varepsilon_j}$ at the beginning of the proof, that we did not relabel to ease notation.
So in fact up to now we just showed that there exists a subsequence $y_{\varepsilon_j}$ such that
\begin{align*}
\y{ u - \mathcal{R}^\ast_{\varepsilon_j} u_{\varepsilon_j} }_{L^2}^2, \y{ z - \mathcal{R}^\ast_{\varepsilon_j} z_{\varepsilon_j} }_{L^2}^2, \y{\nabla u - \mathcal{R}^\ast_{\varepsilon_j} \nabla _{\varepsilon_j}  u_{\varepsilon_j}}_{L^2}^2, \y{\nabla z - \mathcal{R}^\ast_{\varepsilon_j} \nabla _{\varepsilon_j}  z_{\varepsilon_j}}_{L^2}^2 \epsarrow 0.
\end{align*}
To complete the proof, we actually need to show convergence of the whole sequence $y_\varepsilon$. This can be concluded by the standard argument. For any subsequence $\varepsilon_k$ we can apply the proof and get a subsequence $\varepsilon_{k_j}$ such that $y_{\varepsilon_{k_j}}$ converges in the sense mentioned above to a solution of the limit rate-independent system. By uniqueness of solutions we can conclude that any subsequence $y_{\varepsilon_k} $ admits a sub-subsequence that converges to the unique limit $y$ and thus the whole sequence $y_\varepsilon$ converges in the sense of \eqref{eq:thm:rate}.

\end{proof}

\subsection*{Outlook}
The mathematical framework established in this work opens promising avenues for future research, both from theoretical and numerical perspectives.
The incorporation of gradient plasticity regularization and the hardening tensor in our model provides a natural foundation for capturing complex material behavior. 
One possible extension would be to include viscoplastic effects by incorporating rate-dependent terms in the dissipation potential $\Psi$, transforming it from a purely rate-independent to a unified rate-dependent formulation.
This extension would enable the modeling of strain-rate sensitivity observed in biological tissues under diverse loading conditions—from high-speed impact scenarios to long-term creep deformation.

From a numerical standpoint, 
the existence result for the discrete systems (\Cref{th:existence of discr solutions}) ensures that for any mesh size $\varepsilon > 0$, the discrete elastoplastic system admits an energetic solution. The existence of continuum solutions that is shown in \Cref{th:existence_of_solutions_continuum_system} establishes that the homogenization limit is well-posed. 
The convergence result (\Cref{thm:rate}) provides rigorous theoretical foundations for developing finite element discretizations of the continuum limit while preserving the essential mathematical structure. 
Moreover, the discrete-to-continuum operator framework offers a systematic approach for constructing structure-preserving numerical schemes that maintain the energetic formulation and stability properties of the continuous problem.
Thus, numerical simulations could be performed to gain a deeper understanding into the effective mechanical properties of fiber reinforced materials.  

\newpage
\section*{Appendix - Auxiliary Lemmas}

\begin{lemma}[Lipschitz continuity of convex functions, see Th. 10.4. in \cite{Rockafellar+1970}] \label{lemma:convex_function_Lipschitz}
Let $f$ be a proper convex function and let $S$ be any closed bounded subset of $\operatorname{ri} (\text{dom} f)$. Then $f$ is Lipschitzian relative to $S$, i.e. there exists a real number $L\geq 0$ sucht that
\begin{align*}
    \Y{f(x)-f(y)} \leq L \Y{x-y} \qquad \forall x,y \in S.
\end{align*}
\end{lemma}
Here $\operatorname{ri}(\cdot)$ denotes the \textit{relative interior} of a set and \textit{proper} means that $-\infty <  f(x)$ for every $x$ and $f(x) < + \infty$ for at least one $x$.

\begin{lemma}[see Lemma 4.2 in \cite{lindqvist1990equation}]\label{lemma:lindqvist}
    Let $p \geq 2$ and $w_i \in \mathbb{R}^d$ for $i=1,2$, then
    \begin{align}\label{eq:lindqvist}
        \Y{w_2}^p \geq \Y{w_1}^p + p \Y{w_1}^{p-2}w_1 \cdot (w_2 - w_1) + \frac{\Y{w_2 -w_1}^p}{2^{p-1}-1}.
    \end{align}
\end{lemma}

\begin{lemma}
   Let $v:\Omega\subset\mathbb{R}^{d} \rightarrow \mathbb{R}^{d \times d} $ and $\bold{A}$ be a positive and symmetric fourth-order tensor, then for any $x \in \Omega$ the inequality
   \begin{align}\label{eq_eigenvalues}
       \lambda_{\bold{A}}^\mathrm{min} \Y{v(x)}^2 \leq v(x) : \bold{A}(x) :v(x) \leq \lambda_{\bold{A}}^\mathrm{max} \Y{v(x)}^2
   \end{align}
   is valid, where $\lambda_{\bold{A}}^\mathrm{min}$ and $\lambda_{\bold{A}}^\mathrm{max}$ are the smallest and largest eigenvalues of $\bold{A}$, respectively.
\end{lemma}
\begin{proof}
    As a direct consequence of the Courant–Fischer–Weyl min-max principle we have
    \begin{align*}
        \lambda_{\bold{A}}^\mathrm{min} \leq R_\bold{A} (v) \leq \lambda_{\bold{A}}^\mathrm{max},
    \end{align*}
    where $R_\bold{A}$ is the Rayleigh quotient defined as
    \begin{align*}
        R_\bold{A}v = \frac{v:\bold{A}:v }{\Y{v}^2}.
    \end{align*}
\end{proof}

\begin{lemma}[see Section 3 in \cite{Bihari}] \label{lemma:LaSalle}
    Let $Y(x),F(x)$ be positive continuous functions in $a\leq x \leq b$, let $k,M \geq 0$ and let $w(u)$ be a non-negative non-decreasing continuous function for $u\geq 0$. Then the inequality
    \begin{align*}
        Y(x) \leq k+ M \int_a^x F(t)w(Y(t)) \dif t, \quad a\leq x \leq b
    \end{align*}
    implies the inequality 

    \begin{align*}
        Y(x) \leq \Psi^{-1}\left( \Psi(k) + M \int_a^x F(t) \dif t \right), \quad a \leq x \leq b ,
    \end{align*}
    where 
    \begin{align*}
        \Psi(u) = \int_{u_0}^u \frac{\dif t }{w(t)}, \quad u_0>0, u\geq 0.
    \end{align*}
\end{lemma}

\subsubsection*{Acknowledgments}
This work was funded by the Deutsche Forschungsgemeinschaft (DFG, German Research Foundation) under Germany’s Excellence Strategy – EXC-2193/1 – 390951807.

\newpage
\end{document}